\documentclass[a4paper]{amsart}
\usepackage{a4wide}
\usepackage[dvips]{graphicx}
\usepackage{color}
\usepackage{amssymb}
\usepackage{textcomp}

\usepackage{mathtools}

\let\oldtocsection=\tocsection

\let\oldtocsubsection=\tocsubsection

\let\oldtocsubsubsection=\tocsubsubsection

\renewcommand{\tocsection}[2]{\hspace{0em}\oldtocsection{#1}{#2}}
\renewcommand{\tocsubsection}[2]{\hspace{1em}\oldtocsubsection{#1}{#2}}
\renewcommand{\tocsubsubsection}[2]{\hspace{2em}\oldtocsubsubsection{#1}{#2}}

\usepackage{amsmath}
\usepackage{amsfonts}
\usepackage{amscd}
\usepackage{graphicx}
\usepackage[latin1]{inputenc}
\usepackage[english]{babel}
\usepackage{enumerate}
\usepackage{afterpage}
\newcounter{notes}

\def\a{\alpha}
\def\b{\beta}

\def\g{\gamma}
\def\l{\lambda}

\newcommand{\N}{\mathbb{N}}

\renewcommand{\to}{\longrightarrow}
\def\co{\colon\thinspace}

\newcommand{\cT}{\mathcal{T}}
\newcommand{\C}{\mathbb{C}}

\newcommand{\HH}{\mathbb{H}}

\newcommand{\R}{\mathbb{R}}
\newcommand{\PSL}{\mathrm{PSL}}
\newcommand{\ML}{\mathrm{ML}}
\newcommand{\PGL}{\mathrm{PGL}}

\newcommand{\Sc}{\mathcal{S}}

\newtheorem{Theorem}{Theorem}[section]
\newtheorem{Lemma}[Theorem]{Lemma}
\newtheorem{Proposition}[Theorem]{Proposition}
\newtheorem{Corollary}[Theorem]{Corollary}

\newtheorem{Remark}[Theorem]{Remark}

\newtheorem{introthm}{Theorem}

\begin{document}

\title{On type-preserving representations of the thrice punctured projective plane group}

\author{Sara Maloni}
\address{Department of Mathematics, University of Virginia}
\email{sm4cw@virginia.edu}
\urladdr{www.people.virginia.edu/$\sim$sm4cw}

\author{Fr\'{e}d\'{e}ric Palesi}
\address{Aix Marseille Universit\'{e}, CNRS, Centrale Marseille, I2M, UMR 7373, 13453 Marseille, France}
\email{frederic.palesi@univ-amu.fr}
\urladdr{www.latp.univ-mrs.fr/$\sim$fpalesi}

\author{Tian Yang}
\address{Department of Mathematics, Texas A\&M University}
\email{tianyang@math.tamu.edu}
\urladdr{www.math.tamu.edu/$\sim$tianyang}

\thanks{The first author was partially supported by the National Science Foundation grants DMS 1506920 and DMS 1650811. The third author is supported by the National Science Foundation grants DMS 1405066 and DMS 1812008. Our collaboration was greatly facilitated by support from the GEAR network (U.S. National Science Foundation grants DMS 1107452, 1107263, 1107367 ``RNMS: GEometric structures And Representation varieties'').}

\begin{abstract}
In this paper we consider type-preserving representations of the fundamental group of the three--holed projective plane into $\mathrm{PGL}(2, \R) =\mathrm{Isom}(\HH^2)$ and study the connected components with non-maximal euler class. We show that in euler class zero for all such representations there is a one simple closed curve which is non-hyperbolic, while in euler class $\pm 1$ we show that there are $6$ components where all the simple closed curves are sent to hyperbolic elements and $2$ components where there are simple closed curves sent to non-hyperbolic elements. This answer a question asked by Brian Bowditch. In addition, we show also that in most of these components the action of the mapping class group on these non-maximal component is ergodic. In this work, we use an extension of Kashaev's theory of decorated character varieties to the context of non-orientable surfaces. 
\end{abstract}

\maketitle

\setcounter{tocdepth}{2}
\tableofcontents

\section{Motivations and statements of results} \label{intro}

In this paper we study type-preserving representations of a (possibly non-orientable) punctured surface $S$ into $\rm{Isom} (\HH) = \PGL(2,\R),$ the group of isometries of the hyperbolic space $\HH^2.$ A representation $\rho\co\pi_1(S)\to \PGL(2, \R)$ is said \textit{type-preserving} if peripheral elements are mapped to parabolic isometries and $1$--sided [resp. $2$--sided] elements are mapped to orientation reversing [resp. preserving] isometries. (Recall that an element of $\pi_1 (S)$ is called $2$--sided if it is represented by a curve admitting an orientable regular neighborhood, and $1$--sided otherwise. For an orientable surface, all curves are $2$-sided.) The $\PGL(2,\R)$-character variety of $S$ is the geometric invariant theory quotient $\{\mbox{type-preserving } \rho \co \pi_1 (S) \to \PGL (2 , \R)\} / \PGL (2 , \R),$ and roughly speaking is the space of type-preserving representations, up to conjugation by $\PGL(2, \R)$. In this paper we will work on the space $\mathfrak{X}(S) = \{ \mbox{type-preserving } \rho \co \pi_1 (S) \to \PGL (2 , \R) \} / \PSL (2 , \R)$ of  representation up to conjugation by $\PSL (2 , \R).$  One can obtain the $\PGL(2,\R)$--character variety of $S$ from $\mathfrak{X}(S)$ as a further quotient, which identifies certain connected components, but still maintains their geometrical dichotomy, as discussed in Section \ref{geom}. The reason to consider representations up to $\PSL (2 , \R)$ is because this will allow a unified treatment with the well known case of representations of orientable punctured surfaces in $\PSL(2, \R)$. In particular, the notion of euler classes and the lengths coordinates  extend naturally from $\PSL(2,\R)$--character varieties to $\mathfrak{X}(S).$ We are interested in questions of Kashaev on the number of connected components of $\mathfrak{X}(S)$, of Bowditch on the existence of representations such that the image of all $2$--sided simple closed curves is hyperbolic but which are not discrete and faithful and of Goldman on the ergodicity of the action of the mapping class group $\mathrm{Mod}(S)$ on $\mathfrak{X}.$ We will answer these three questions in the case of the thrice-punctured projective plane $N_{1,3}$ respectively  in Theorem \ref{conncomp}, Theorem \ref{eul_pm1} and Theorem \ref{ergodicity}.

The main tool we use are length coordinates for the decorated character variety. Originally, the length coordinates were defined on the decorated Teichm\"uller space of orientable punctured surfaces by Penner\,\cite{Penner} and generalized by Kashaev\,\cite{kas_coor} to the entire character space of type-preserving representations of orientable surfaces. In this paper, we generalized the coordinates to the case of non-orientable surfaces, and found a formula of computing the trace of curves in terms of this coordinates, following ideas of Roger-Yang\,\cite{roger-yang}. Another main tool in the proofs above is the choice of a good ideal triangulation of the surface, which we call balanced triangulations, and were introduced and used by Huang--Norbury\,\cite{hua_sim} and by the first and the second authors in \cite{mal_ont2}.

In the following paragraphs we will see many instances of a dichotomy between closed and punctured surfaces. On the other hand, orientable and non-orientable surfaces seem to behave similarly, but this is not always the case. As an example we want to discuss the growth of simple closed curves. A famous result of Mirzakhani\,\cite{mir_gro} shows that for an orientable surface $S$ the number of closed geodesics of a given topological type of length less than $L$ is asymptotically equivalent to a positive constant times $L^{\mathrm{dim}(\ML(S))}$, where $\ML(S)$ is the space of measured laminations of $S.$ In \cite{gen_pay, hua_sim, gam_ana, mag_cou} it was proved that this is not true any more for non-orientable surfaces, and the difference comes from the different dynamics of the action of $\mathrm{Mod}(S)$ on $\ML(S)$, as Gendulphe\,\cite{gen_wha} explains. 

\subsection{Number of connected components}

For each type-preserving representation $\rho$, one can define its Euler class $e(\rho)$ as its representation area divided by $2\pi$, see Section \ref{sec:euler} for a more detailed discussion. It is know that the Euler class satisfies the Milnor-Wood inequality $\chi(\Sigma_{g,n})\leqslant e(\rho)\leqslant -\chi(S)$, see \cite{mil_ont, wood_bundles}, and Goldman\,\cite{goldman:topological} (in the orientable case) and Palesi\,\cite{palesi_connected} (in the non-orientable case) proved that the equality holds if and only if $\rho$ is Fuchsian, that is, discrete and faithful. If we consider closed surfaces, Goldman\,\cite{goldman:topological} proved that the Euler class defines a one-to-one correspondence between the connected components of $\mathfrak{X}(\Sigma_g)$ and the integers $e$ with $|e|\leqslant -\chi(\Sigma_g)$, and Palesi\,\cite{palesi_connected} extended the result to the case of closed non-orientable surfaces $N_k.$  

For a punctured orientable surface $\Sigma_{g,n},$ the number of connected components of $\mathfrak{X}(\Sigma_{g,n})$ is more subtle to describe since for an integer $e$ with $|e|\leqslant -\chi(S),$ the spaces $\mathfrak X_e(S)$ of (conjugacy classes of) type-preserving representations of Euler class $e$ can either be empty or non-connected, see \cite{kas_coor}. In the orientable case Kashaev\,\cite{kas_coor} conjectured that it should be determined by the Euler class and an extra invariant which corresponds to the $\PSL(2,\R)$--conjugacy classes of the holonomy representations of the boundary elements. More precisely, a parabolic element in $\PSL(2,\R)$ is, up to $\pm I$, conjugate to an upper-triangular matrix, and its conjugacy class is distinguished by whether the the sign of the nonzero off diagonal element is positive or negative. We respectively call the two conjugacy classes of parabolic elements the positive and the negative conjugacy classes. For a type-preserving $\rho\co\pi_1 (S) \rightarrow \PGL(2,\R),$ we say that the \emph{sign} of a puncture $v$ is  \emph{positive} (resp. \emph{negative}), denoted by $s(v)=+1$ (resp. $s(v)=-1$), if $\rho$ sends a peripheral element around this puncture into a positive (resp. negative) conjugacy class of parabolic elements. For $s\in\{\pm 1\}^n,$ we denote by $\mathfrak{X}^s_e(S)$ the space of conjugacy classes of type-preserving representations with Euler class $e$ and signs of the punctures $s.$ It is conjectured in \cite{kas_coor} that each $\mathfrak{X}^s_e(S)$ is either empty or connected. The third author\,\cite{yan_ont} proved this conjecture in the case of the four holed sphere $\Sigma_{0,4}$ and in this article we prove it in the case of the thrice-punctured projective plane.

\begin{introthm}\label{conncomp}
  Let $s \in \{\pm 1\}^3.$
  \begin{enumerate}
    \item $\mathfrak{X}_0^s(N_{1, 3})$ is nonempty if and only if $s$ contains exactly one or two $+1$'s.
    \item $\mathfrak{X}_{+ 1}^s(N_{1, 3})$ is nonempty if and only if $s$ contains exactly two or three $+1$'s.
    \item $\mathfrak{X}_{- 1}^s(N_{1, 3})$ is nonempty if and only if  $s$ contains exactly two or three $-1$'s.
    \item All the nonempty spaces above are connected.
  \end{enumerate}
\end{introthm}
As a consequence, $\mathfrak{X}_{0}(N_{1, 3})$ has six connected components, while $\mathfrak{X}_{+ 1}(N_{1, 3})$ and $\mathfrak{X}_{- 1}(N_{1, 3})$ each has four connected components. A surprising fact that we will discuss below is that different connected components will have different geometric properties.

\subsection{Hyperbolicity of simple closed curves}\label{geom}

Bowditch\,\cite{bow_mar} asked the following question: \textit{Given a non-elementary type-preserving representation $\rho\co\pi_1(S)\to \PGL(2, \R)$, is it true that if $\rho$ sends every non-peripheral $2$--sided simple closed curve to an hyperbolic element of $\PSL(2,\R)$, then $\rho$ is Fuchsian?} (Recall that a representation $\rho$ is called {\em non-elementary} if its image is Zariski-dense in $\PGL(2,\R).$) March\'e and Wolff\,\cite{mar-wol} answered affirmatively for the genus $2$ surface $\Sigma_2$, but it was proven in \cite{mal_ont} and \cite{yan_ont} that answer is no, in some cases for the four-times punctured sphere $\Sigma_{0,4}.$ Our next result discusses exactly for which components of $\mathfrak{X}(N_{1, 3})$ the answer is yes, giving a complete answer to Bowditch's question in this case.
\begin{introthm}\label{eul_pm1}\noindent
\begin{enumerate}
\item If $s \in \{\pm 1\}^3$ contains exactly two $+1$'s, then every type-preserving $\rho$ in a full measure subset of $\mathfrak X_1^s(N_{1,3})$ sends every non-peripheral simple closed curve to a hyperbolic element.
\item  If $s \in \{\pm 1\}^3$ contains exactly two $-1$'s, then every type-preserving $\rho$ in a full measure subset of  $\mathfrak X_{-1}^s(N_{1,3})$ sends every non-peripheral simple closed curve to a hyperbolic element.
\item  Let $s_+=(+1,+1,+1)$ and let $ s_-=(-1,-1,-1).$ Then every representation in $\mathfrak{X}_{1}^{s+}(N_{1, 3})$ and $\mathfrak{X}_{-1}^{s_-}(N_{1, 3})$ sends some non-peripheral $2$-sided simple closed curve to a non-hyperbolic element. 
\item Every non-elementary type-preserving representation  $\rho\co \Gamma_{1,3} \to \PGL(2,\R)$ with relative Euler class $e(\rho)= 0$ sends some non-peripheral simple closed curve to a non-hyperbolic element.
\end{enumerate}
 \end{introthm}

In particular, the representations in $\mathfrak{X}_{1}^{s+}(N_{1, 3})$, $\mathfrak{X}_{-1}^{s_-}(N_{1, 3})$ and $\mathfrak{X}_{0}(N_{1, 3})$ are not Fuchsian, so Theorem \ref{eul_pm1} gives a negative answer to Bowditch's question for these eight components. On all the other components, the answer to Bowditch's question is affirmative.

\subsection{Ergodicity of the mapping class group action}

The pure (extended) mapping class group $\mathrm{Mod}(S)$ is the group of isotopy-classes of homeomorphisms of $S$ fixing the boundary components point-wise. It naturally acts on $\mathfrak X(N_{k,n})$ preserving the Euler class $e$ and the sign of the boundary holonomy $s.$ In the case of closed oriented surfaces Goldman\,\cite{goldman-survey} conjectured that this action is ergodic on each non-extremal and non-zero component. March\'e and Wolff\,\cite{mar-wol} proved that a positive answer to Bowditch's question implies Goldman conjecture and used this to prove Goldman conjecture for $\Sigma_2.$ In the case of punctured surfaces, since Bowditch's Conjecture is no longer true for all the connected components, the proof of Goldman's result is more difficult.  The third author\,\cite{yan_ont} proved it that for the four-times punctured sphere and in this article we prove it in most cases of the thrice-punctured projective plane.

\begin{introthm}\label{ergodicity}\noindent
  \begin{enumerate}
    \item The mapping class group $\mathrm{Mod}(N_{1,3})$ acts ergodically on the connected component $\mathfrak{X}_{1}^{s+}(N_{1, 3}).$ 
    \item The mapping class group $\mathrm{Mod}(N_{1,3})$ acts ergodically on the connected component $\mathfrak{X}_{-1}^{s-}(N_{1, 3}).$  
    \item The mapping class group $\mathrm{Mod}(N_{1,3})$ acts ergodically on every connected component of $\mathfrak{X}_{0}(N_{1, 3}).$  
  \end{enumerate}
\end{introthm}

The ergodicity of the action of $\mathrm{Mod}(N_{1,3})$ on the components $\mathfrak{X}_{1}^{s}(N_{1, 3})$ with $s \in \{\pm 1\}^3$ containing exactly two $+1$'s,  and on the components $\mathfrak{X}_{-1}^{s_-}(N_{1, 3})$ with $s \in \{\pm 1\}^3$ containing exactly two $-1$'s, is still unknown and deserves further study. We tend to believe that the action is still ergodic.

\subsection{Domination of representations} 

As a side result, we also show the following result about domination of representations, an analog of the results of Gueritaud-Kassel-Wolff\,\cite{G-K-W} and Deroin-Tholozan\,\cite{deroin-tholozan} for orientable closed surfaces and the third author\,\cite{yan_ont} for orientable punctured surfaces. Recall that a representation $\rho$ is said to be \textit{dominated} by another representation $\rho'$ if the traces of the simple closed curves of $\rho$ are less than or equal to those of $\rho'$, in absolute value.

\begin{introthm}\label{domination}
Given a non-orientable punctured surface $N_{k, n}$, every non-Fuchsian type-preserving representation is dominated by a Fuchsian one and almost every Fuchsian representation dominates at least one representation with Euler class $e$ such that $|e|<-\chi(N_{k,n}).$
\end{introthm}

\subsection{Organization of the paper}

In Section \ref{sec:back} we will discuss type-preserving representations of punctured surfaces and generalize Penner--Kashaev's theory of length coordinates and trace formulas to the case of non-orientable surfaces. At the end of the section we will prove Theorem \ref{domination}. In Section \ref{back_tri} we will discuss balanced triangulations of the thrice punctured projective plane $N_{1,3}$ and the triangle switches. In Section \ref{sec:conn} we will discuss connected components of the character variety $\mathfrak{X}(N_{1, 3})$ and prove Theorem \ref{conncomp}, while in Sections \ref{eul_1} and \ref{eul_0} we will discuss representations with Euler class $\pm 1$ and $0$, respectively, and prove Theorems \ref{eul_pm1}. Finally in Section \ref{sec:ergodicity}, we will discuss the ergodicity of action of the mapping class group on the non-maximal components of $\mathfrak{X}(N_{1, 3})$ and prove Theorem \ref{ergodicity}. 

\section{Type-preserving representations of punctured surfaces} \label{sec:back}

In this section we recall the necessary background on decorated representation spaces for type-preserving representations of punctured hyperbolizable surfaces, which was originally defined for orientable surfaces, and we extend it to the non-orientable setting.

\subsection{Euler class}\label{sec:euler}

We let a punctured hyperbolizable surface $S$ be either  an orientable surface $\Sigma_{g, n}$ of genus $g$ and $n$ punctures with $2-2g-n<0$ or a non-orientable surface $N_{k,n}$ of genus $k$ (meaning that it is the connected sum of $k$ crosscaps) and $n$ punctures with $2-k-n<0.$ Given a type-preserving representation $\rho\co \pi_1 (S) \to\PGL(2,\R)$, a {\em pseudo-developing map} associated to $\rho$ is a piecewise smooth $\rho$--equivariant map $\mathrm{D}_\rho\co\widetilde{S}\to \HH^2.$ 

In the case $S=\Sigma_{g,n}$, we define the {\em (relative) Euler class} as
$$e(\rho) := \frac{1}{2\pi}\int_{\Sigma_{g, n}} (\mathrm{D}_\rho)^\star \omega,$$ 
where $\omega$ is the hyperbolic area form on $\HH^2.$ The Euler class is an integer, and the Milnor--Wood inequality $|e(\rho)|\leqslant -\chi(\Sigma_{g, n})$ bounds the absolute value of $e(\rho)$ by the absolute value of the Euler characteristic of $\Sigma_{g, n}$, see Milnor\,\cite{mil_ont} and Wood\,\cite{wood_bundles}. A classical result of Goldman\,\cite{goldman:topological} shows that the representations with maximal Euler class are discrete and faithful representations, and the set of such representations consists of two connected components corresponding to the Teichm\"uller space of the surface $\Sigma_{g,n}$, one for each orientation of $\Sigma_{g,n}.$

If  $N = N_{k, n}$ is a non-orientable surface and $\rho \co \pi_1 (N_{k,n}) \to \PGL(2,\R)$ is a type-preserving representation, then it defines an orientable circle bundle, and one can define the Euler class of $\rho$, as the Euler class of the associated circle bundle as in \cite{wood_bundles}. We can also use the following alternative and equivalent definition, which uses the orientable double cover $p \co \widehat{N}\to N$ of $N$, and is easier in this context. The surface $\widehat{N}$ is homeomorphic to the orientable surface $\Sigma_{k-1, 2n}$ and the fundamental group $\pi_1 (\widehat{N})$ is an index two subgroup of $\pi_1 (N)$ which corresponds to the subset of elements represented by $2$--sided curves. If $\rho$ is a type-preserving representation, then all elements corresponding to $2$--sided curves are sent to orientation-preserving isometries. Hence, the restriction of $\rho$ to $\pi_1 (\widehat{N}) \subset \pi_1 (N)$ defines a type-preserving representation $\widehat{\rho} \co \pi_1 (\widehat{N}) \to \PSL(2,\R).$ We define the Euler class of the representation $\rho$ as 
$$e(\rho) := \frac{1}{2} e(\widehat{\rho}).$$
In this case, the Euler class is still an integer and the Milnor-Wood inequality still applies so that $|e(\rho)| \leqslant - \chi(N_{k,n})$, see Wood\,\cite{wood_bundles}. As in the orientable case, the representations with maximal Euler class correspond to discrete and faithful representations and the set of such representations consists of two connected components corresponding to the Teichm\"uller space of $N.$ This definition of the Euler class of a type-preserving representation of a non-orientable surface is equivalent to another definition which uses obstruction classes and which has been discussed by the second author in \cite{palesi_connected}.

\subsection{Kashaev's coordinates on decorated character varieties}\label{back_kas}

\subsubsection{Admissible triangulations}

Let $S$ be a punctured hyperbolizable surface. An ideal arc $\gamma$ in $S$ is  called $\rho$--{\em admissible} if the endpoints of $D_\rho(\widetilde{\gamma})$, where $\widetilde{\gamma}$ is a lift of $\gamma$ to the universal cover $\widetilde S$ of $S$,  are distinct. (One can check that this definition depends neither on the choice of the lift $\widetilde{\gamma}$ of $\gamma$, nor on the choice of $D_\gamma.$) An (ideal) triangulation of $S$ is a disjoint union of ideal arcs such that their complement in $S$ are ideal triangles. A triangulation of $S$ is called $\rho$--{\em admissible} if all its arcs are. This property is preserved by conjugation.

We recall the following result of Kashaev: 

\begin{Theorem}[Kashaev\,\cite{kas_coor}] \label{thm_kash}
  For each ideal triangulation $\cT$ of a punctured hyperbolizable $S$, the set $$\mathfrak{X}_{\cT}(S) = \{[\rho] \in \mathfrak{X}(S) \mid \cT\text{ is } \rho-\text{admissible}\}$$ of representations for which $\cT$ is admissible is open and dense in $\mathfrak{X}(S)$ and there exist finitely many triangulations $\cT_1, \cdots, \cT_r$ such that the associated spaces $\mathfrak{X}_{\cT_i}(S)$ cover $\mathfrak{X}(S).$ 
\end{Theorem}

Kashaev only stated this theorem for orientable surfaces, but the original proof does not rely on the orientability of the surface, hence the theorem is still valid for all punctured hyperbolizable surfaces, including non-orientable ones.

\subsubsection{Decorated character varieties and Kashaev's coordinates}

Let $S = \Sigma_{g,n}$ or $N_{k,n}$ and let $\Gamma$ be its fundamental group. Given a type-preserving representation $\rho\co \Gamma \to \PGL(2,\R)$, a {\em decoration} $d$ of $\rho$ is a choice of $\rho(\Gamma)$--invariant horocycles, one for each fixed point of the image under $\rho$ of the peripheral elements of $\Gamma.$ Two decorated representations $(\rho_1, d_1)$ and $(\rho_2, d_2)$ are equivalent if $\rho_2 = g\rho_1 g^{-1}$ and $d_2 = g \cdot d$ for some $g\in \PSL(2, \R).$ We define the {\em decorated character variety} $\mathfrak{X}^d(S)$ to be the space of equivalence classes of decorated representations.

For any $\rho$--admissible ideal triangulation $\cT$ of $S$ with edge set $E$ and triangle set $T$, and for any pseudo-developing map $D_\rho\co\widetilde{S} \to \HH^2$, the lengths coordinates of $(\rho, d)$ consist of two types of coordinates: the $\lambda$--lengths of the edges and the signs of the triangles, defined as follows. For any edge $e$, let $\widetilde{e}$ be a lift of $e$ to $\widetilde{S}$ and let $H_1$ and $H_2$ the two horocycles associated to the two endpoints of $D_\rho(\widetilde{e}).$ We define the $\lambda$--\text{length} of $e$ to be
$$\lambda(e)=\exp\left(\frac{l(e)}{2}\right),$$ 
where $l(e)$ is the signed distance between $H_1$ and $H_2$ (that is, $l(e)>0$ if the horocycles are disjoint and $l(e)\leqslant 0$ otherwise). We can see that this gives a well-defined map $$\lambda\co E \to \R_{>0}.$$
Given a $\rho$--admissible triangulation $\cT$ of an orientable surface $\Sigma_{g,n}$, we assign a sign to each triangle $t \in \cT$ as follows. Let $(v_1, v_2, v_3)$ be the three vertices of $t$ so that the orientation on $t$ given by their cyclic order agrees with the orientation of $t$ given by the orientation on $\Sigma_{g, n}.$ Let $\widetilde{t}$ be a lift of $t$ to the universal cover $\widetilde{\Sigma}_{g, n}$ and let $\widetilde{v}_i$ be a lift of the vertices $v_i.$ Then $D_\rho(\widetilde{v}_1)$, $D_\rho(\widetilde{v}_2)$ and $D_\rho(\widetilde{v}_3)$ determine an ideal triangle $\Delta$ in $\HH^2.$ We define the \textit{sign} $\epsilon(t) = 1$ if the orientation of $\Delta$ given by the cyclic order of $(D_\rho(\widetilde{v}_1), D_\rho(\widetilde{v}_2), D_\rho(\widetilde{v}_3))$ coincides with the one induced by the orientation of $\HH^2.$ Otherwise, we define $\epsilon(t) = -1.$ One can check that this is well-defined and that the relative Euler class can then also be calculated as 
\begin{equation}\label{euler_sign}
  e(\rho) = \frac{1}{2}\sum_{t\in \cT}\epsilon(t).
\end{equation}

For a non-orientable surface $N_{k,n}$, we let $\cT$ be a $\rho$--admissible triangulation  on $N_{k,n}.$ Then this triangulation lifts to a triangulation $\widehat{\cT}$ of its orientable double cover $\widehat{N}_{k,n}$, where each arc has two lifts, and this triangulation is $\widehat{\rho}$--admissible. We fix an orientation on $\widehat{N}_{k,n}$ once and for all. So if $t$ is a triangle in $\cT$, then it lifts to two distinct triangles $\hat{t}$ and $\hat{t}'$ in the double cover. One can assign signs to these two triangles using $\widehat{\rho}$, as explained above. 

\begin{Proposition}
  For a type-preserving representation $\rho$ and a triangle $t$, the signs of the two lifted triangles $\hat{t}$ and $\hat{t}'$ are equal.
\end{Proposition}

\begin{proof}
  Let $N$ and $\widehat{N}$ be defined as before, then we have $N = \widehat{N}/\sigma,$ where $\sigma$ is an orientation-reversing homeomorphism of $\widehat{N}$, so that $\sigma (\hat{t}) = \hat{t'}.$ If the cyclic order of the vertices of $\hat{t}$ induced by the orientation on $\widehat{N}$ is $(v_1 , v_2 , v_3)$, then for the triangle $\hat{t}' = \sigma (\hat{t})$, the cyclic order becomes $(v'_1, v'_3 , v'_2)$ where $v'_i = \sigma (v_i).$ Without loss of generality, we can assume that $\varepsilon(\hat{t} ) = 1$ so that $(D_{\hat{\rho}}(\widetilde{v}_1), D_{\hat{\rho}}(\widetilde{v}_2), D_{\hat{\rho}} (\widetilde{v}_3))$ is in the cyclic order given by the orientation of $\mathbb{H}^2.$ Now, consider the other triangle $(D_{\hat{\rho}}(\widetilde{v}'_1), D_{\hat{\rho}}(\widetilde{v}'_2), D_{\hat{\rho}} (\widetilde{v}'_3)).$ By $\rho$-equivariance, this triangle is the image of the first triangle by $\rho (\gamma)$ for a certain $1$--sided curve $\gamma.$ As the representation is type-preserving, the isometry $\rho (\gamma)$ is orientation reversing, and hence the second triangle has the opposite cyclic order. This proves that $\varepsilon (\hat{t}') = 1 = \varepsilon (\hat{t} ).$ 
\end{proof}

Using this Proposition, we can define the sign of the triangle $t$ in the original triangulation to be the common sign of both lifted triangles for the lifted representation, so we get a well-defined map $$\epsilon\co T \to \{-1, +1\}.$$ 
The Euler class of the lifted representation is then given by 
$$e(\widehat{\rho}) =  \frac{1}{2}\sum_{\hat{t} \in \hat{\cT}}\epsilon(\hat{t})  = \frac {1}{2} \left (2 \sum_{t\in \cT}\epsilon(t) \right),$$ so we get the same formula as in Equation \eqref{euler_sign} for the Euler class of a type-preserving representation. By using the orientable double cover $p\co\Sigma_{0, 6}\to N_{1, 3}$ and Kashaev's result \cite[Theorem 2]{kas_coor} for orientable surfaces we get the following.

\begin{Theorem}\label{coordin}
There exists a principal $((\R_{>0})^V)$--bundle $p\co \mathfrak{X}^d(S)\to \mathfrak{X}(S)$ such that for any triangulation $\mathcal{T}$ of $S$ we have that $$\mathfrak{X}^d_{\cT} (S) = p^{-1}(\mathfrak{X}_{\cT} (S)) = \bigsqcup_{\epsilon \in \{\pm 1\}^T}\mathfrak{X}^d_{\cT}(\epsilon),$$ and, via the $\l$--lengths, $\mathfrak{X}^d_{\cT}(\epsilon)$ is isomorphic as principal $((\R_{>0})^V)$--bundle to an open subset of $(\R_{>0})^E$ defined as the complement of the zeros of certain rational function coming from the image of the peripheral elements not being the identity matrix.
\end{Theorem}

\subsubsection{Diagonal switches}

Given a triangulation $\cT$ we can define a new triangulation $\cT'$ by replacing an edge $e$ with the other diagonal $e'$ inside the quadrilateral formed by the two triangles adjacent to $e.$ This operation is called a {\em diagonal switch}. Any two triangulations are related to each other by finitely many diagonal switches, see for example Harer\,\cite{harer_virtual} in the orientable case, and Negami\,\cite{neg_diagonal} in the non-orientable case. The following result describes what happens to $\rho$--admissibility and $\lambda$--lengths when we do this move. Let $t_1$ and $t_2$ the triangles in $\cT$, and $t_1'$ and $t_2'$ be the triangles in $\cT'$ and let these triangles have edges (in cyclic order) $t_1= (e, e_1, e_2)$, $t_2=(e, e_3, e_4)$, $t_1'=(e', e_3, e_1)$ and $t_2'=(e', e_2, e_4).$
 
\begin{Theorem}
  Suppose $\cT$ is $\rho$--admissible. Then:
\begin{enumerate}
  \item If $\epsilon(t_1)=\epsilon(t_2)$, then $\cT'$ is $\rho$--admissible and, in this case, the signs and lengths of the common triangles of $\cT$ and $\cT'$ are unchanged, and we have $$\epsilon(t_1')= \epsilon(t'_2)=\epsilon(t_1) \text{ and } \lambda(e') = \frac{\lambda(e_1)\lambda(e_3) + \lambda(e_2)\lambda(e_4)}{\lambda(e)}.$$
  \item If $\epsilon(t_1)\neq\epsilon(t_2)$, then $\cT'$ is $\rho$--admissible if and only if $\lambda(e_1)\lambda(e_3) \neq \lambda(e_2)\lambda(e_4).$ In this case, the signs and lengths of the common triangles of $\cT$ and $\cT'$ are unchanged, and we have
  \begin{enumerate}[(a)]
    \item If $\lambda(e_1)\lambda(e_3) < \lambda(e_2)\lambda(e_4)$, then $$\epsilon(t'_1)=\epsilon(t_1), \;\epsilon(t'_2)=\epsilon(t_2) \text{ and } \lambda(e') = \frac{\lambda(e_2)\lambda(e_4) - \lambda(e_1)\lambda(e_3)}{\lambda(e)}.$$
    \item If $\lambda(e_1)\lambda(e_3) > \lambda(e_2)\lambda(e_4)$, then $$\epsilon(t'_1)=\epsilon(t_2),\; \epsilon(t'_2)=\epsilon(t_1) \text{ and } \lambda(e') = \frac{\lambda(e_1)\lambda(e_3) - \lambda(e_2)\lambda(e_4)}{\lambda(e)}.$$
  \end{enumerate}
\end{enumerate}
\end{Theorem}

  The previous result was only stated for orientable surface in \cite{yan_ont}. However, as this corresponds to local relations inside ideal quadrilateral, one can also apply it to non-orientable surfaces directly.

\subsection{Trace formulas for closed curves}\label{back_trace}

In this section we recall the formula for the trace of the image of a curve in function of the coordinates described above. The formula was stated in \cite{yan_ont} from an idea of Sun and Yang. They discussed the case of orientable surfaces, but their formula can be extended to any $2$--sided simple closed curve on a surface, independent of the orientation of the surface, and we will discuss at the end of this section how to extend the result to $1$--sided simple closed curves.

Let $\gamma$ be an immersed closed $2$--sided curve, and choose an orientation on the orientable regular neighborhood of $\gamma.$ Homotopy it, if necessary, to its {\em standard position}, that is, in such a way that $\gamma$ intersects each triangle tranversely crossing from one edge $e_1$ to another edge $e_2.$ Let $e_3$ be the third edge of $t.$ 

For each triangle $t$ we define the matrix $M(t)$ as follows. If $\gamma$ makes a left turn as in (a) of Figure \ref{fig:lr}, then let $M(t) =   \begin{bmatrix} 
    \lambda(e_1) & \epsilon(t)\lambda(e_3)\\ 
    0 & \lambda(e_2) 
  \end{bmatrix}$, while if $\gamma$ makes a right turn as in (b) of Figure \ref{fig:lr}, then let $M(t) =   \begin{bmatrix} 
      \lambda(e_2) & 0\\ 
      \epsilon(t)\lambda(e_3) & \lambda(e_1) 
    \end{bmatrix}.$

    \begin{figure}
[hbt] \centering
\includegraphics[height=4 cm]{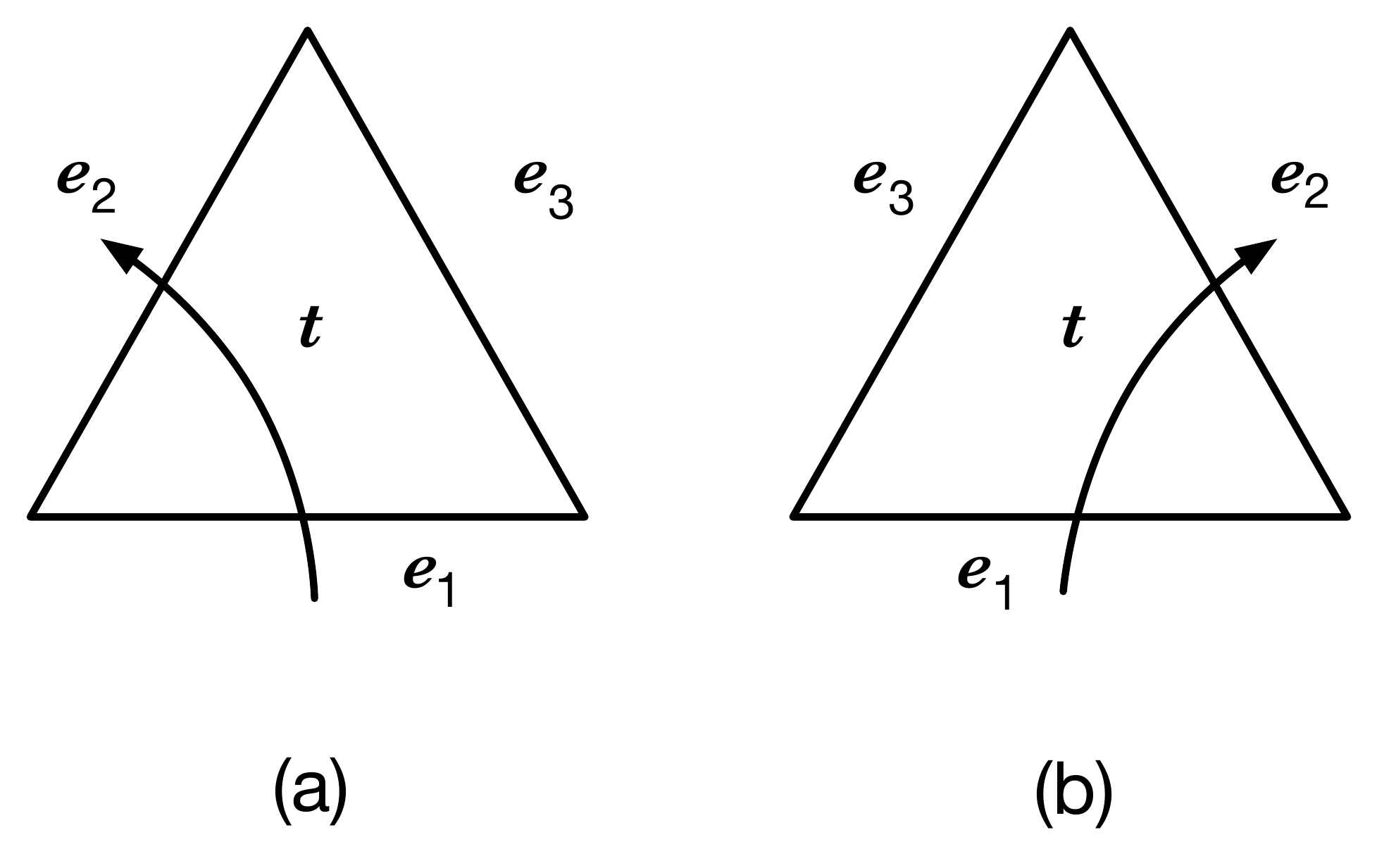}
\caption{}
\label{fig:lr}
\end{figure}
    
    From this, we get the following formula.

\begin{Theorem}[Sun--Yang\,\cite{yan_ont}]\label{SunYang}
  For an immersed $2$--sided closed curve $\gamma$ in standard position, let $e_{i_1}, \ldots, e_{i_m}$ and $t_{j_1}, \ldots, t_{j_m}$ be the edges and triangles, respectively, intersected by $\gamma$ in the cyclic order following the orientation of $\gamma$ so that $e_{i_k}$ is the common edge of $t_{j_{k-1}}$ and $t_{j_{k}}.$ Then
\begin{equation}\label{trace_formula}
  |\mathrm{tr}\rho([\gamma])| = \frac{|\mathrm{tr}\left(M(t_{j_1})\cdots M(t_{j_m})\right)|}{\lambda(e_{i_1})\cdots \lambda(e_{i_m})}.
\end{equation}
\end{Theorem}

Note that the formula giving the trace of $\rho (\gamma)$ does not depend on the choice of orientation on the regular neighborhood of $\gamma.$

\begin{Remark}For an immersed  $1$--sided closed curve $\gamma$, one can compute the trace of the $2$--sided curve $\gamma^2$ using the Theorem above, and then use the formula 
\begin{equation}\label{square}
(\mathrm{tr}\rho([\gamma]))^2 = |\mathrm{tr}\rho([\gamma^2])|-2.
\end{equation}
\end{Remark}

We now have all the tools to prove Theorem \ref{domination}. More precisely, we prove the following.

\begin{Theorem}\label{dominance} 
\begin{enumerate}[(1)]
\item For every non-Fuchsian type-preserving representation $\rho:\Gamma_{g,n}\rightarrow \PGL(2,\mathbb R),$ there exists a Fuchsian type-preserving representation $\rho'$ such that 
$$\big|tr\rho([\gamma])\big|\leqslant\big|tr\rho'([\gamma])\big|$$  
for each $[\gamma]\in\Gamma_{g,n},$ and the strict inequality holds for at least one $\gamma.$

\item Conversely, for almost every Fuchsian type-preserving representation $\rho':\Gamma_{g,n}\rightarrow \PGL(2,\mathbb R)$ and for each $e$ with $|e|<n+k-2$ and $\mathfrak X_e(N_{k,n})\neq\emptyset,$ there exists a type-preserving representation $\rho$ with $e(\rho)=e$ such that
$$\big|tr\rho([\gamma])\big|\leqslant\big|tr\rho'([\gamma])\big|$$  
 for each $[\gamma]\in\Gamma_{g,n},$ and the strict inequality holds for at least one $\gamma.$
\end{enumerate} 
\end{Theorem}

\begin{proof} For (1), by Section \ref{back_kas}, there exists a $\rho$-admissible ideal triangulation $\mathcal T.$ Choose arbitrarily a decoration $d$ of $\rho,$ and let $(\rho',d')$ be the decorated representation that has the same $\lambda$-lengths of $(\rho,d)$ and positive signs for all the triangles. Since $\rho'$ is maximal, it is Fuchsian by \cite{palesi_connected}. Let $\gamma$ be a $2$-sided curve. Applying Formula (\ref{trace_formula}) to $|tr\rho([\gamma])|$ and $|tr\rho'([\gamma])|,$ we see that they have the same summands with different coefficients $\pm 1$, and the coefficients for the second one are all positive. Since each summand is a product of the $\lambda$-lengths, which is positive, the inequality follows. Since $\rho$ is non-Fuchsian, by (\ref{euler_sign}), there must be an ideal triangle $t$ that has negative sign in $(\rho,d).$ Therefore, if $\gamma$ intersects $t$, then some of the summands in the expression of $|tr\rho([\gamma])|$ have negative coefficients, and the inequality for $\gamma$ is strict. Now for a $1$-sided curve $\gamma,$ the same argument above shows that we have the same comparison for $\gamma^2.$ Then the  results follows from Formula (\ref{square}).

For (2), choose arbitrarily an ideal triangulation $\mathcal T$ of $N_{k,n},$ and let $T$ be the set of ideal triangles of $\mathcal T.$ By Section 2.2, if $\mathfrak X_e(N_{k,n})\neq\emptyset,$ then there exists $\epsilon\in\{\pm 1\}^T$ such that $\sum_{t\in T}\epsilon(t)=2e$ and the subset $\mathfrak{X}_{\mathcal T}(\epsilon) = p(\mathfrak{X}_{\mathcal T}^d(\epsilon))$ is homeomorphic via the lengths coordinate to a full measure open subset of $\mathbb R_{>0}^E.$ For each $\lambda$ in such an open subset, let $(\rho,d)$ be the decorated representation determined by $(\lambda,\epsilon).$ Then $e(\rho)=e.$ On the other hand, $\mathbb R_{>0}^E$ is identified with the decorated Teichm\"uller space via the lengths coordinate, hence $\lambda$ determines a Fuchsian type-preserving representation $\rho'.$ By the same argument in (1), the inequality holds for $\rho$ and $\rho'
,$ and is strict for $\gamma$ intersecting the ideal triangles $t$ with $\epsilon(t)=-1.$
\end{proof}

\section{Balanced triangulations of $N_{1,3}$}\label{back_tri}

Let $N = N_{1, 3}$ be a (topological) thrice-punctured (real) projective plane. We denote by $\Sc = \Sc(N)$ the set of free homotopy classes of essential simple closed curves on $N.$  Recall that a curve is \emph{essential} if it does not bound a disc, a punctured disk, an annulus or a M\"obius strip. Since most of the curve we will consider are essential, we will often omit the word essential. 

\begin{figure}
[ht] \centering
\includegraphics[height=5 cm]{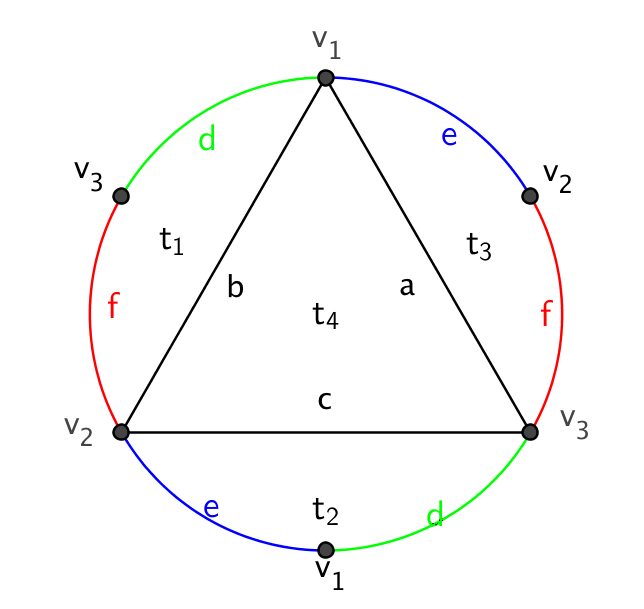}
\includegraphics[height=5 cm]{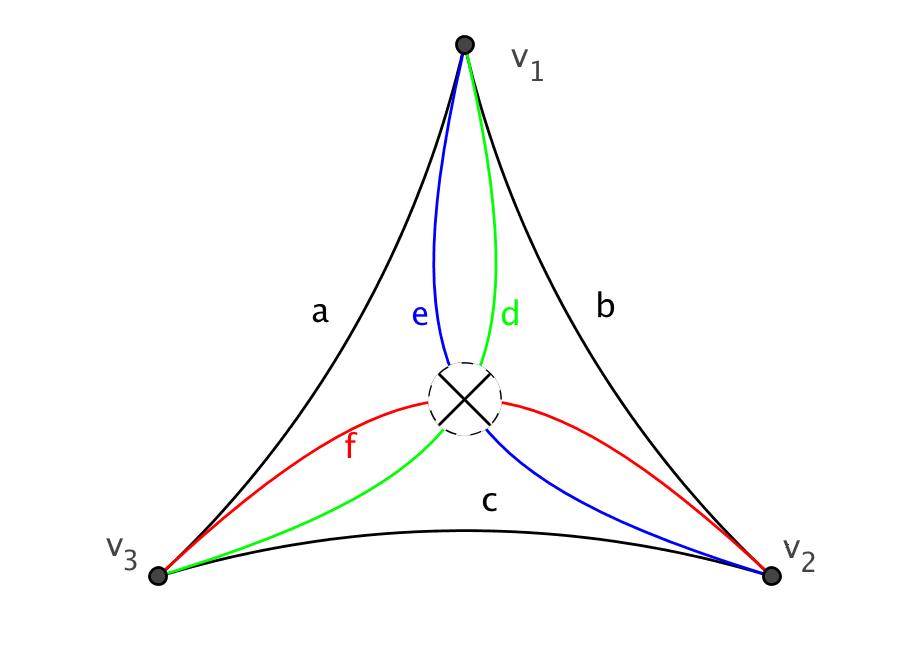}
\caption{Balanced triangulations of $N_{1, 3}$, where $E = \{a, b, c, d, e, f\}$ and $T = \{t_1, t_2, t_3, t_4\}.$}
\label{fig:triang}
\end{figure}

\begin{Proposition}\label{rem1}
  There is a canonical correspondence between:
    \begin{itemize}
   \item (unordered) pairs of (free homotopy classes of) distinct $1$--sided simple closed curves intersecting exactly once;  
  	\item (free homotopy classes of essentials) $2$--sided simple closed curves; 
  	\item (isotopy classes of) arc $e$ joining two different punctures.
  \end{itemize}
\end{Proposition}

\begin{proof}
The correspondence between the first two sets goes as follows. Given an (unordered) pair $(\a, \b)$ of (free homotopy classes of) $1$--sided simple closed curves intersecting exactly once, there is a unique essential $2$--sided simple closed curve disjoint from $\a \cup \b.$ Conversely, any (free homotopy class of) $2$--sided simple closed curve splits $N$ into two connected components: a pair of pants and a projective plane $M$ with one puncture and one hole which contains exactly two $1$--sided simple closed curves intersecting once. 

For the correspondence between the second and the third sets, we can see that to each arc $e$ joining two different punctures, we can associate the unique $2$--sided simple closed curve $\gamma_e$ going around the two punctures joined by $e$ which does not intersect $e.$ We will say that $\gamma_e$ is {\em associated} with $e.$ Conversely, any (free homotopy class of) $2$--sided simple closed curve splits $N$ into a projective plane with one puncture and one hole and a pair of pants (with two punctures and one hole). This pair of pants contains exactly one arc between the two punctures. See Figure \ref{fig:onesided}.
\end{proof}

\begin{figure}
[hbt] \centering
\includegraphics[height=4 cm]{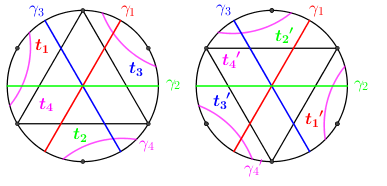}
\caption{The $4$ one sided simple closed curves $\gamma_1, \gamma_2, \gamma_3, \gamma_4$ associated with the triangles in $\cT = \{t_1, t_2, t_3, t_4\}$ (on the left) and the $4$ one sided simple closed curves $\gamma_1, \gamma_2, \gamma_3, \gamma'_4$ associated with the triangles in the new traingulations $\cT' = \{t_1', t_2', t_3', t'_4\}$ obtained from $\cT$ after a triangle switch $T_4.$}
\label{fig:onesided}
\end{figure}

We now define certain triangulations of $N_{1,3}$ which will be important in the rest of the paper. A triangulation $\cT$ of $N_{1,3}$ is called {\em balanced} if each edge of $\cT$ joins two distinct punctures; or, equivalently, if each triangle of $\cT$ has vertices in the three (distinct) punctures $v_1$, $v_2$ and $v_3$ of $N_{1,3}.$ Let $a,b,c,d,e,f$ be the edges of the triangulation described in Figure \ref{fig:triang}. We have four triangles $t_1, t_2, t_3, t_4$ corresponding to triangles with edges $t_1= \{b,d,f\}$ , $t_2 = \{c,d,e\}$, $t_3 = \{a,e,f\}$ and $t_4 = \{a,b,c\}$ respectively. There are two edges between the vertex $v_1$ and $v_2$ (the edges $a$ and $d$), similarly we have two edges between $v_2$ and $v_3$ (the edges $b$ and $e$) and two edges between $v_3$ and $v_1$ (the edges $c$ and $f$).

\begin{figure}
[hbt] \centering
\includegraphics[width=\textwidth]{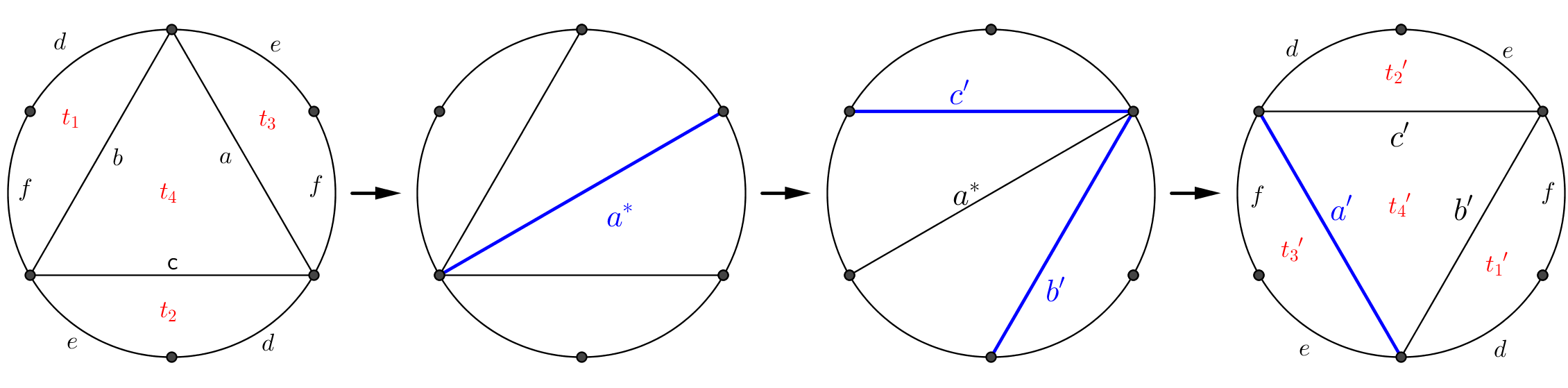}
\caption{The four edge flips that generate a triangle switch.}
\label{tri_switch}
\end{figure}

There are elementary moves between balanced triangulations called \emph{triangle switches}, see Figure \ref{tri_switch}.  Given a triangle $t$ in a balanced triangulation, one can cut the surface along the three edges that do not belong to this triangle. (For example, in the case of Figure \ref{fig:triang}, if $t = t_4 =\{a, b, c\}$, one should cut along the edges $\{d, e, f\}.$) This gives an ideal hexagon with the triangle $t$ in the interior. The action of a triangle switch on the balanced triangulation is given by turning the triangle $t$ `upside down' inside this ideal hexagon to get the triangle $t'.$ Each other triangle $t_i$ of the triangulation for $i = 1, 2, 3$, is sent to the triangle $t'_i$ which is `opposite' to $t_i$, see Figure \ref{tri_switch}. These triangle switches can be expressed as a sequence of elementary flips of the triangulation. Using the notations of Figure \ref{tri_switch}, the switch is equivalent to the following sequence of elementary moves:
\begin{itemize}
  \item Flip the edge $a$ into $a^*;$
  \item Flip edges $b$ into $c'$ and $c$ into $b';$ (those two flips commute)
  \item Flip $a^*$ into $a';$
\end{itemize}
The triangles of the new triangulation $\cT' =\{t'_1, t'_2, t'_3, t'_4\}$ are the following:  $t'_1= \{b',d,f\}$ , $t'_2 = \{c',d,e\}$, $t'_3 = \{a',e,f\}$ and $t'_4 = \{a',b',c'\}.$ By symmetry of the switch in the arcs $a,b,c$, we will also get the same final triangulation when flipping $b$ or $c$ first.

The balanced triangulation graph $\mathfrak{T}$ of $N$ is the abstract simplicial graph whose vertices are isotopy classes of balanced triangulations of $N$ and whose edges are defined by triangle switches. This graph is the one skeleton of the simplicial complex dual to the balanced arc complex.  We will show that $\mathcal{T}(N)$ is isomorphic, as abstract simplicial graph, to the one skeleton of the complex $\Upsilon$ defined in \cite{maloni-palesi}, which is the simplicial complex whose $k$--simplices of $\Upsilon$ are given by subsets of $4-k$ distinct (isotopy classes of) $1$--sided simple closed curves in $N$ that pairwise intersect once. Note that $\Upsilon$ is the simplicial dual to the complex of $1$-sided curves $\mathcal{CC}(N)$ of $N$, see Scharlemann\,\cite{sch_the}. (Recall that an isomorphism $f\co X \to Y$ between abstract simplicial complexe $X$ and $Y$ is a bijection $f\co \mathrm{Vert}(X) \to \mathrm{Vert}(Y)$ such that $x \in X \iff f(x) \in Y.$)  In \cite{maloni-palesi} and \cite{hua_sim}, it was shown that the one-skeleton of the simplicial complex $\Upsilon$ is the Cayley graph of a finite index subgroup of the mapping class group generated by four involutions $\theta_1, \ldots, \theta_4$, as we will discuss and use in Section \ref{ergodicity2}.

\begin{Proposition}\label{equiv_tri}
  There is an isomorphism $f\co \mathfrak{T} \to \Upsilon.$
\end{Proposition}

Since $\Upsilon$ is connected, we can see that there is  a unique sequence of triangle switches between any two balanced triangulations.

\begin{proof} 
  First, let's define the bijection between $\mathrm{Vert}(\mathfrak{T})$ and $\mathrm{Vert}(\Upsilon).$ This is already discussed in Huang--Norbury \cite[Lemma 6]{hua_sim}. Let $\cT$ be a balanced triangulation and $t_i$ one of its triangles.  The surface $N_{1,3} \setminus t_i$ is a one-holed projective plane, which has only one non-trivial simple closed curve. Hence there is a unique $1$--sided simple closed curve on the surface that does not intersect the triangle $t_i.$ Using this, we can see that a balanced triangulation gives rise to a quadruple of one-sided curves $\alpha_1 , \dots , \alpha_4$, and that these four curves pairwise intersect exactly once. Reciprocally, let $\alpha_1 , \dots , \alpha_4$ be a quadruple of $1$-sided simple closed curves pairwise intersecting once. For each choice of two curves $\alpha_i$ and $\alpha_j$, the surface $N \setminus (\alpha_i \cup \alpha_j) $ is a disjoint union of a punctured disc and a twice punctured disc. So to each pair of curve, one can associate the arc joining the two punctures in the twice punctured disc. This gives $6$ disjoint arcs that form a balanced triangulation.

Second, let's see that $e \in \mathrm{Edges}(\mathfrak{T})$ if and only if $f(e) \in \mathrm{Edges}(\Upsilon)$
	Let $\cT = (t_1 , \dots , t_4)$ be a balanced triangulation, and $\cT' =(t'_1 , \dots , t'_4)$ be the balanced triangulation obtained from $\cT$ after a triangle switch along $t_i.$ Let $(\alpha_1 , \dots , \alpha_4)$ and $(\alpha'_1 , \dots , \alpha'_4)$ be the corresponding $4$-tuples of $1$--sided simple closed curves associated to these two triangulations. We have that $\alpha_j = \alpha'_j$ for all $j \neq i$ as one can seen in Figure \ref{fig:onesided}, so the two sets of $(\alpha_1 , \dots , \alpha_4)$ and $(\alpha'_1 , \dots , \alpha'_4)$ define a triple of $1$--sided simple closed curves pairwise interecting once. Conversely, using Proposition \ref{rem1} we can see that a triple $(\alpha_i, \alpha_j, \alpha_k)$ of $1$--sided simple closed curves pairwise interecting onceis associated to a triple of arcs joining two different punctures and so an ideal triangle $t$ with vertices in the three punctures. We can then associate to $(\alpha_i, \alpha_j, \alpha_k)$  the triangle switch $S_t.$
\end{proof}

\begin{figure}
[hbt] \centering
\includegraphics[height=5 cm]{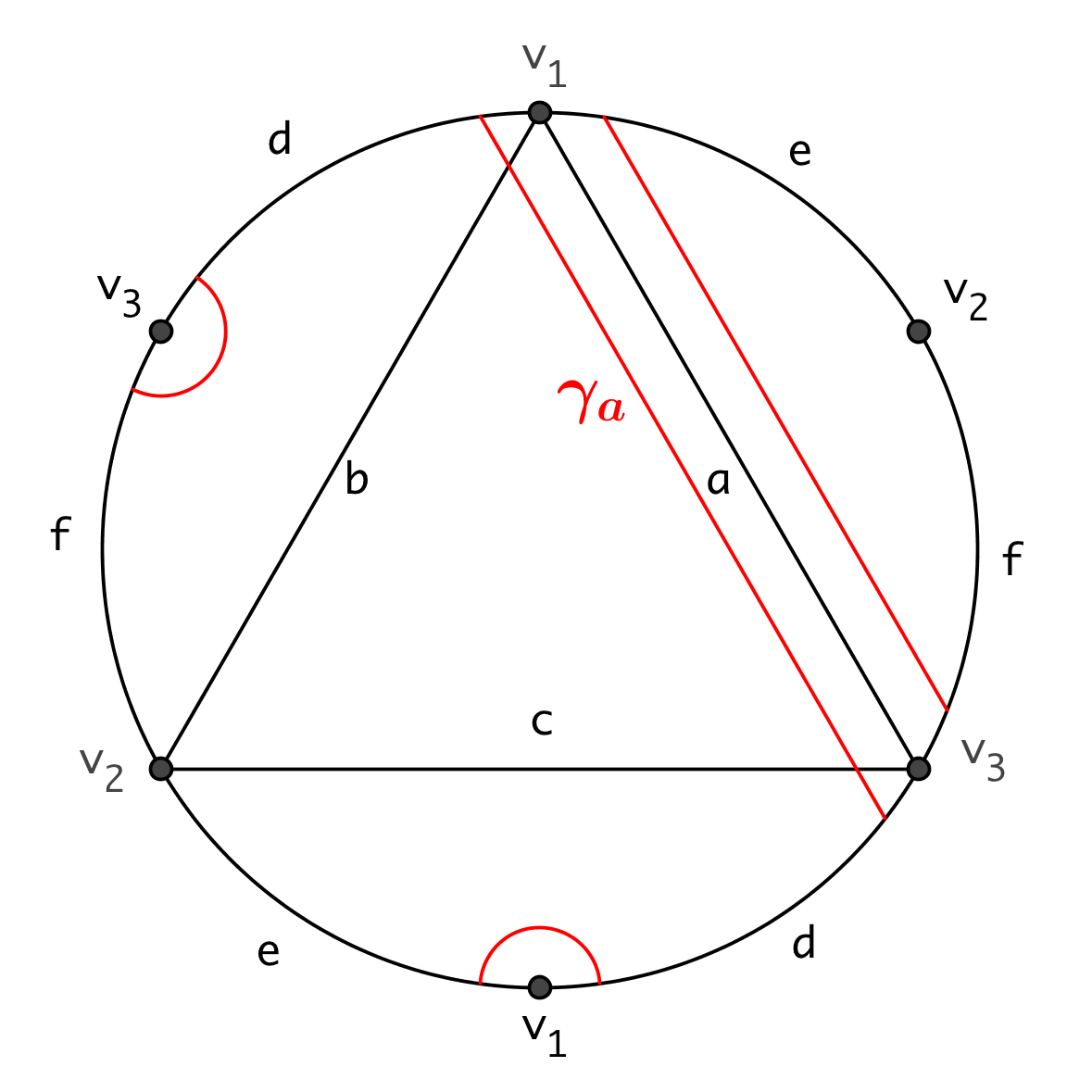}
\caption{The curve $\gamma_a$ associated to the edge $a$ of the triangulation $\cT$ of $N_{1, 3}.$}
\label{curves_asso}
\end{figure}

\section{Connected components of $\mathfrak{X}(N_{1, 3})$}\label{sec:conn}

In this section we will prove Theorem \ref{conncomp}, using a new parametrization of the space $\mathfrak{X}(N_{1, 3})$ by so-called \emph{triangle coordinates}.

\subsection{Triangle coordinates} 

Let $\rho$ be a type-preserving representation of $\pi_1(N_{1,3})$ and let $d$ be a decoration of $\rho.$ Suppose $\cT$ is a $\rho$--admissible balanced triangulation of $N_{1,3}$ and let $V$, $E$ and $T$ be the set of vertices, edges and triangles of $\cT$, respectively. 

Given a decorated representation $[(\rho, d)]\in \mathfrak{X}^d_{+1}(N_{1,3})$, let $(\lambda, \epsilon) = (\lambda ([(\rho, d)]), \epsilon ([(\rho, d)])) \in \R_{>0}^E \times \{\pm 1\}^T$ be the $\lambda$--lengths of the edges and the signs of the triangles. For a triangle $t_i \in \cT$ whose edges are given by $e_1 , e_2 , e_3$, we define the \emph{$i$-th triangle coordinate} as the product $\lambda (e_1) \lambda (e_2) \lambda (e_3).$

More precisely, let $\cT = \{t_1, t_2, t_3, t_4\}$ be the triangulation described in Figure \ref{fig:triang}, where $t_1= \{b,d,f\}$ , $t_2 = \{c,d,e\}$, $t_3 = \{a,e,f\}$ and $t_4 = \{a,b,c\}.$ The triangle coordinates are defined by the following quantities.
\begin{itemize}
  \item $X_1 = \lambda(b)\lambda(d)\lambda(f)$, corresponding to the triangle $t_1;$
  \item $X_2 = \lambda(c)\lambda(d)\lambda(e)$, corresponding to the triangle $t_2;$
  \item $X_3 = \lambda(a)\lambda(e)\lambda(f)$, corresponding to the triangle $t_3;$
  \item $X_4 = \lambda(a)\lambda(b)\lambda(c)$, corresponding to the triangle $t_4.$
\end{itemize}

Theorem \ref{thm_kash} proves that the set $\mathfrak{X}_{\cT}(N_{1, 3})$ of representations $[\rho] \in \mathfrak{X}(N_{1, 3})$ such that $\cT$ is $\rho$--admissible is an open and dense subset of $\mathfrak{X}(N_{1, 3});$ see Section \ref{back_kas}. The following result is an analog of Lemma 6.2 of \cite{yan_ont} in our context, proving that the projectivized coordinates $[X_1 , \dots , X_4]$, together with the signs parameters, parametrize the components of $\mathfrak{X}_{\cT}(N_{1, 3}).$

\begin{Lemma}\label{lem_phi}
  Given $\mu \in (\R_{>0})^V$ and $\lambda \in (\R_{>0})^E$, let $(\R_{>0})^E$ be the principal $\left((\R_{>0})^V\right)$--bundle given by $(\mu \cdot \lambda) (e_{i j}) = \mu(v_i) \lambda(e_{i j}) \mu(v_j)$, and let $(\R_{>0})^4$ be the principal $(\R_{>0})$--bundle given by $r \cdot (X_1, X_2, X_3, X_4) = (r X_1, r X_2, r X_3, r X_4).$

Then the map 
\begin{align*}
	\Phi \co (\R_{>0})^E &\longrightarrow  (\R_{>0})^4 \\ (\lambda(e_{1,2}), \cdots, \lambda(e_{3,4})) &\longmapsto (X_1, \cdots, X_4) \end{align*}
	  induces a diffeomeorphism $$\Phi^\ast \co (\R_{>0})^E / (\R_{>0})^V \to (\R_{>0})^4 /\R_{>0}.$$
\end{Lemma}

\begin{proof}
	For any $\mu \in (\R_{>0})^V$ and $\lambda \in (\R_{>0})^E$, we have $\Phi(\mu \cdot \lambda) = (\prod_{i = 1}^{3}(\mu(v_i))^2) \Phi(\lambda) = r \cdot \Phi (\lambda)$, and hence  $\Phi^\ast$ is well-defined. 
  
  Moreover, we have that
	 $$\Phi\left(\displaystyle\frac{1}{(x_1 x_2 x_3 x_4)^{\frac 13 } } (x_3 x_4, x_1 x_4, x_2 x_4, x_1 x_2, x_2 x_3, x_1 x_3 )\right) = (x_1, x_2, x_3, x_4)$$ for all $(x_1, x_2, x_3, x_4) \in  (\R_{>0})^4$, so that $\Phi^\ast$ is surjective.
  
  For the injectivity, suppose that $\Phi(\lambda') = r \cdot \Phi(\lambda)$, then let 
  \begin{enumerate}
    \item[] $v_1(\lambda) = (\lambda(\epsilon_{1, 2})\lambda(\epsilon_{3, 1}))^2 \lambda(\epsilon_{2,3})\lambda(\epsilon_{2,1})\lambda(\epsilon_{3,2})\lambda(\epsilon_{1,3});$
    \item[] $v_2(\lambda) = (\lambda(\epsilon_{1, 2})\lambda(\epsilon_{2,3}))^2 \lambda(\epsilon_{3,1})\lambda(\epsilon_{2,1})\lambda(\epsilon_{3,2})\lambda(\epsilon_{1,3});$
    \item[] $v_3(\lambda) = (\lambda(\epsilon_{2,3})\lambda(\epsilon_{3, 1}))^2 \lambda(\epsilon_{1,2})\lambda(\epsilon_{2,1})\lambda(\epsilon_{3,2})\lambda(\epsilon_{1,3}).$
  \end{enumerate}
  and let $\mu(v_i) = r^{-2} \displaystyle\frac{v_i(\lambda')}{v_i(\lambda)}.$ Then $\lambda'(e_{i j}) = \mu(v_i) \lambda(e_{i j}) \mu(v_j)$, and so $\Phi^\ast$ is injective.
  
  Finally, the differentiability of $\Phi^\ast$ and $(\Phi^\ast)^{-1}$ follows from the definition of $\Phi.$
\end{proof}

Let $\cT$ be a balanced triangulation of $N.$ Since we can identify $(\R_{>0})^4/\R_{>0}$ with the simplex 
$$\Delta = \{ (x_1 , x_2 , x_3 , x_4) \in \R_{>0} \, | \, x_1 + x_2 + x_3 + x_4 = 0 \},$$ 
we have that $\mathfrak{X}_{\cT} (N_{1,3}) = \bigsqcup_{\epsilon \in \{\pm 1\}^{T}}\mathfrak{X}_{\cT}(\epsilon)$ and $\Psi_{\cT}$ induces maps
$\mathfrak{X}_{\cT} (N_{1,3}) \rightarrow \Delta$ which are diffeomorphism on their image. These images are open subset of $\Delta$ defined as the complement of the zeros of certain rational functions, and we will discuss them in next section.

\subsection{Holonomy of cusps}

Let $v_1$, $v_2$ and $v_3$ be the three punctures of $N_{1,3}.$ If $\pi_1(N_{1,3}) = \langle \g_1, \g_2, \g_3 \rangle$ with $\g_i$ the generators described in Figure \ref{fig:onesided}, then the peripheral simple closed curves going around $v_1$, $v_2$, and $v_3$ correspond to the elements $[\g_1\g_2]$, $[\g_2\g_3]$ and $[\g_3\g_1]$, respectively. The images of these $2$--sided peripheral simple closed curves are parabolic elements in $\PSL(2, \R)$, and, up to conjugation, they are represented by matrices $\pm \begin{bmatrix} 
  1 & x\\ 
  0 & 1 
\end{bmatrix}$ where $x\in \R$ and where the sign of $x$ is well-defined.

We can express the holonomy of these curves using the triangle coordinates. The following result is a simple computation using the formulas of Section \ref{back_trace}.

\begin{Lemma}[Holonomy of cusps]\label{cusp} Let $\rho$ be a representation with coordinates $(X , \epsilon) \in \Delta \times \{ -1 , 1 \}^4.$
  Up to conjugation, the $\rho$--images of the peripheral elements $ [\alpha\beta],   [\beta\gamma]$ and $ [\alpha\gamma] \in \pi_1(N_{1,3})$ are:
  \begin{equation*}
    \begin{aligned}
      \rho([\g_1\g_2])&= \pm \begin{bmatrix} 
        1 & \epsilon(t_2)X_1+\epsilon(t_1)X_2+\epsilon(t_4)X_3+\epsilon(t_3)X_4\\ 
        0 & 1 
      \end{bmatrix}\\    
      \rho([\g_2\gamma_3])&= \pm \begin{bmatrix} 
        1 & \epsilon(t_3)X_1+\epsilon(t_4)X_2+\epsilon(t_1)X_3+\epsilon(t_2)X_4\\ 
        0 & 1 
      \end{bmatrix}\\  
      \rho([\gamma_3\g_1])&= \pm \begin{bmatrix} 
         1 & \epsilon(t_4)X_1+\epsilon(t_3)X_2+\epsilon(t_2)X_3+\epsilon(t_1)X_4\\ 
         0 & 1 
      \end{bmatrix}  
    \end{aligned}
  \end{equation*}
\end{Lemma}

Recall that $\epsilon$ determines the Euler class of the representation as $ \displaystyle e(\rho) = \sum_{i=1}^4 \epsilon (t_i).$ So for  $\{i, j, k, l\} = \{1, 2, 3, 4\}$, let $\epsilon_{i}$ and $\epsilon_{i, j}$ in $\{-1, +1\}^T$ be defined as follows
  \begin{itemize}
	\item $\epsilon_{i}(t_i) = - 1$ and $\epsilon_{i}(t_k) = + 1$ for all $k \neq i;$
    \item $\epsilon_{i, j}(t_i) = \epsilon_{i, j}(t_j) = -1$ and $\epsilon_{i, j}(t_k) = \epsilon_{i, j}(t_l) = 1.$
  \end{itemize}

So we get that:
	\begin{itemize}
		\item $e(\rho) = 0$ if and only if $\epsilon = \epsilon_{i,j}$ for some $\{i,j\} \subset \{ 1, 2 , 3 , 4\}$, 
		\item  $e(\rho) = + 1$ if and only if $\epsilon = \epsilon_{i}$ for some $i \in \{1,2,3,4\}$, and
		\item  $e(\rho) = - 1$ if and only if $\epsilon = - \epsilon_{i}$ for some $i \in \{1,2,3,4\}.$
	\end{itemize}

For the representation $\rho$ to be type-preserving, the terms appearing in these matrices need to be non-zero. So we define the following subsets of $\Delta$: 
	\begin{enumerate}
    \item[] $\Delta^{i, j} = \Delta^{k, l} = \{(x_1, x_2, x_3, x_4) \in \Delta \mid x_i + x_j \neq x_k+ x_l\}$, and
		\item[] $\Delta^{\overline{i}} = \{(x_1, x_2, x_3, x_4) \in \Delta \mid x_j \neq x_i+ x_k+ x_l, x_k \neq x_i+ x_j+ x_l, x_l \neq x_i+ x_j+ x_k\}.$
    \end{enumerate}

We now obtain the following result as a direct consequence of Theorem \ref{thm_kash}, and  Lemmas \ref{lem_phi} and \ref{cusp}. 

\begin{Corollary}\label{corol_comp}
Given a balanced triangulation $\cT$ of $N_{1, 3}$ we have: 
  \begin{enumerate}
    \item[] $\mathfrak{X}_{\cT}(N_{1, 3}) \cap \mathfrak{X}_{0}(N_{1, 3}) = \displaystyle \coprod_{i < j = 1}^4 \Delta^{i, j} \times \{ \epsilon_{i,j} \};$
    \item[] $\mathfrak{X}_{\cT}(N_{1, 3}) \cap \mathfrak{X}_{+1}(N_{1, 3}) = \displaystyle \coprod_{i = 1}^4 \Delta^{\overline{i}} \times \{ \epsilon_{i} \};$
    \item[] $\mathfrak{X}_{\cT}(N_{1, 3}) \cap \mathfrak{X}_{-1}(N_{1, 3}) = \displaystyle \coprod_{i = 1}^4 \Delta^{\overline{i}} \times \{ - \epsilon_{i} \}.$
  \end{enumerate}
\end{Corollary}

\subsection{Change of coordinates}\label{change}

In this section, following the notation from the previous section, we will state the formulas for the change of coordinates when we do a triangle switch. We need to distinguish the case of Euler class $e(\rho) = 0$ and of Euler class $e(\rho) = \pm 1.$ The formulas are straightforward computations. 

\subsubsection{Triangle switches in the case $e(\rho) = + 1$}\label{change_e1}

\begin{Lemma}[Change of coordinate when $e(\rho) = 1$]\label{ch_e1}
Suppose $e(\rho) = 1$ and let $ l \in \{ 1 , 2 , 3 , 4 \}.$ We denote by $\cT'$ be the triangulation obtained after a $S_l$-switch, and $(X', \epsilon')$  be the new coordinates. 
		
		The triangulation $\cT'$ is $\rho$-admissible if and only if  we have
		$$\forall \{i,j,k,l\} = \{1 , 2 , 3 , 4 \}, \quad -\epsilon (t_i) X_i + \epsilon (t_j) X_j + \epsilon (t_k) X_k \neq 0.$$
		 In addition we have: 
		$$\begin{dcases}
         X'_i = \displaystyle\frac{|-\epsilon (t_i) X_i + \epsilon (t_j) X_j + \epsilon (t_k) X_k |}{X_l} X_i, \, \mbox{ if } i\neq l\\
         X'_l = \displaystyle \frac{(-\epsilon (t_i) X_i + \epsilon (t_j) X_j + \epsilon  (t_k) X_k)(\epsilon (t_i) X_i - \epsilon (t_j) X_j + \epsilon (t_k) X_k)(\epsilon (t_i) X_i + \epsilon (t_j) X_j - \epsilon (t_k) X_k)}{X_l^2} .
    \end{dcases}$$
	Finally, the signs of the new triangles are given by
		\begin{enumerate}	
			\item If $\epsilon = \epsilon_{l} $, then 
			$\begin{cases} 
				\mbox{ if } \exists i \neq l , X_i > X_j + X_k, \mbox{ then } \epsilon' = \epsilon_i  \\
				\mbox{ otherwise } \epsilon' = \epsilon_l. 
			\end{cases}$
			\item If $\epsilon = \epsilon_k$ with $k \neq l$ then
				$\begin{cases} 
				\mbox{ if }\exists i\neq l,  X_i > X_j + X_k, \mbox{ then } \epsilon' = \epsilon_j  \\
				\mbox{ otherwise } \epsilon' = \epsilon_l.
			\end{cases}$
		\end{enumerate}	 
\end{Lemma}

\subsubsection{Triangle switches in the case $e(\rho) = 0$}\label{change_e0}

\begin{Lemma}[Change of coordinates when $e(\rho) = 0$]\label{ch_e0}
Suppose $e(\rho) = 0$ and let $\{ i,j,k,l \} = \{ 1 , 2 , 3 , 4 \}$ such that $\epsilon_i = \epsilon_j$ and $\epsilon_k = \epsilon_l.$ We denote by $\cT'$ be the triangulation obtained after a $S_l$-switch, and $(X', \epsilon')$  be the new coordinates.

	 The triangulation $\cT'$ is $\rho$-admissible if and only if $ X_k \neq  X_i +  X_j .$ In addition we have   
		$$\begin{dcases}
         X'_i = \displaystyle\frac{|X_i +  X_j -  X_k |}{X_l} X_i, \, \mbox{ if } i\neq l;\\
         X'_l = \displaystyle \frac{( X_i + X_j -  X_k)^3}{X_l^2}. 
    \end{dcases}$$
Finally, if $X_k > X_i + X_j$ then $\epsilon' = \epsilon$, and else $\epsilon' = - \epsilon.$ 
\end{Lemma}

Note that in projective coordinates we get:
$$[X_i', X_j', X_k',X_l'] = \left[ X_i, X_j, X_k, \frac{(X_i+X_j-X_k)^2}{X_l} \right]. $$

\subsection{Proof of Theorem \ref{conncomp} }

In order to describe the connected components of $\mathfrak{X}_{0}(N_{1, 3})$ and $\mathfrak{X}_{\pm 1}(N_{1, 3})$, we need to study the sign of the quantities appearing in the matrices of Lemma \ref{cusp}.  To simplify notations, we let $s_i^{\pm}, s^{\pm} \in \{\pm 1\}^3$ be defined by 
    \begin{itemize}
	\item $s_i^+(v_i) =  +1$, and $s_i^+ (v_j) = -1$ if $j\neq i;$
	\item $s^+(v_i)  = +1 $ for all $i = 1, \ldots, 3;$
	\item $s_i^- = - s_i^+$, and $s^- = - s^+.$
    \end{itemize}

By doing a counting of all the possibilities, we will prove the following result, which is a more precise version of Theorem \ref{conncomp}.

\begin{Theorem}\label{comp} We have the following decompositions:
  
\begin{enumerate}
    \item $\displaystyle \mathfrak{X}_0(N_{1, 3}) = \coprod_{i = 1}^{3} \mathfrak{X}_0^{s_i^{+}}(N_{1, 3}) \coprod_{i = 1}^{3} \mathfrak{X}_0^{s_i^{-}}(N_{1, 3}).$
    \item $\displaystyle \mathfrak{X}_{+1}(N_{1, 3}) = \mathfrak{X}_{+1}^{s^+}(N_{1, 3}) \coprod_{i = 1}^{3} \mathfrak{X}_{+ 1}^{s_i^-}(N_{1, 3}).$
    \item $\displaystyle \mathfrak{X}_{- 1}(N_{1, 3}) = \mathfrak{X}_{- 1}^{s^-}(N_{1, 3}) \coprod_{i = 1}^{3} \mathfrak{X}_{- 1}^{s_i^+}(N_{1, 3}).$
    \item The spaces $\mathfrak{X}_0^{s_i^{+}}(N_{1, 3})$, $\mathfrak{X}_0^{s_i^{-}}(N_{1, 3})$, $\mathfrak{X}_{+ 1}^{s_i^{-}}(N_{1, 3})$, $\mathfrak{X}_{+ 1}^{s^{+}}(N_{1, 3})$, $\mathfrak{X}_{- 1}^{s_i^{+}}(N_{1, 3})$  and $\mathfrak{X}_{- 1}^{s^{-}}(N_{1, 3})$ are connected.
  \end{enumerate}
\end{Theorem}

\begin{proof}
Let $\cT$ be a balanced triangulation of $N_{1,3}$ with set of vertices $V$, and let $s \in \{\pm 1\}^{V}$ and $k \in \{-1, 0, +1\}.$ Since $\mathfrak{X}_{\cT}(N_{1, 3})$ is open and dense in $\mathfrak{X}(N_{1, 3})$, then $\mathfrak{X}_{k}^s(N_{1, 3})\neq \emptyset$ if and only if $\mathfrak{X}_{k}^s(N_{1, 3}) \cap \mathfrak{X}_{\cT}(N_{1, 3}) \neq \emptyset.$ 
\begin{itemize}
  \item For $(1)$, we note that all representations in $\mathfrak{X}_{k} (N_{1, 3}) \cap \mathfrak{X}_{\cT}(N_{1, 3})$ are such that $\epsilon = \epsilon_{i,j}$ for a certain $\{ i , j \} \in \{ 1,2,3,4 \}.$ In that case, there exists $p \in \{1 ,2 , 3 \}$ such that $s = s_p^\pm.$ If we define $$\Delta^{i,j,+} = \Delta^{k, l, -}= \{(x_1, x_2, x_3, x_4) \in \Delta \mid x_i + x_j > x_k+ x_l\},$$ then we get for example that $ (\Delta^{1,2,-} \times \{ \epsilon_{1,2} \}) \cup (\Delta^{3,4,-} \times \{ \epsilon_{3,4} \} ) $ is diffeomorphic to an open dense subset of $\mathfrak{X}_0^{s_1^+} (N_{1,3} ) .$ By symmetry we obtain that 
  $$\displaystyle \mathfrak{X}_0(N_{1, 3}) = \coprod_{i = 1}^{3} \mathfrak{X}_0^{s_i^{+}}(N_{1, 3}) \coprod_{i = 1}^{3} \mathfrak{X}_0^{s_i^{-}}(N_{1, 3}).$$
  \item We use a similar reasoning for $(2).$  All representations in $\mathfrak{X}_{+1} (N_{1, 3}) \cap \mathfrak{X}_{\cT}(N_{1, 3})$ are such that $\epsilon = \epsilon_{i}$ for a certain $ i  \in \{ 1,2,3,4 \}.$ We let $\Delta^{i, +} = \{(x_1, x_2, x_3, x_4) \in \Delta \mid x_i > x_j + x_k+ x_l\}$, and we say that $(x_1 , x_2 , x_3 , x_4)$ satify the \emph{Generalized Triangle Inequality} (GTI) if $(x_1 , x_2 , x_3 , x_4)$ does not belong to any $\Delta^{i,+}.$ In other words, 
  	\begin{equation}\tag{GTI}\label{GTI}  (x_1 , x_2 , x_3 , x_4) \mbox{ satisfies (GTI) } \Leftrightarrow \forall i,j,k,l \in \{1 , 2, 3, 4 \}, \,  x_i \leqslant x_j + x_k + x_l. \end{equation}
Now, for a representation $\rho$ whose coordinates are in $\Delta^{i} \times \{ \epsilon_i \}$ we get the following possibilities~:
  	\begin{itemize}
  		\item $(X_1 , X_2 , X_3  X_4) \in \Delta^{j,+}$ for some $j \neq i$, if and only if $s = s_p^-$ for a certain $p \in \{ 1,2,3 \}.$
  		\item $(X_1 , X_2 , X_3 , X_4) \in \Delta^{i,+}$ or $(X_1 , X_2 , X_3, X_4)$ satisfies \eqref{GTI}, if and only if $s = s^+.$
  	\end{itemize}
Hence we get that $\displaystyle \mathfrak{X}_{+1}(N_{1, 3}) = \mathfrak{X}_{+1}^{s^+}(N_{1, 3}) \coprod_{i = 1}^{3} \mathfrak{X}_{+ 1}^{s_i^-}(N_{1, 3}).$ Part $(3)$ of the theorem is done in the same way as part $(2).$
\item In order to prove $(4)$, we need to prove that all the sets $\mathfrak{X}_{k}^{s}(N_{1, 3})$ appearing in $(1), (2)$ and $(3)$ are connected. By symmetry it is sufficient to prove that $\mathfrak{X}_{0}^{s_1^+}(N_{1, 3})$, $\mathfrak{X}_{+1}^{s_1^-}(N_{1, 3})$ and $\mathfrak{X}_{+1}^{s^+}(N_{1, 3})$ are connected, and we will focus on $\mathfrak{X}_{0}^{s_1^+}(N_{1, 3}).$ 

By Theorem \ref{thm_kash}, Lemma \ref{cusp} and Corollary \ref{corol_comp}, we have that $\Delta^{1,2, -}_{\cT, \epsilon_{1, 2}}\cup \Delta^{3,4, -}_{\cT, \epsilon_{3, 4}}$ is diffeomorphic to a dense subset of $\mathfrak{X}_{0}^{s_1^+}(N_{1, 3}).$ Both $\Delta^{1,2, -} \times \{\epsilon_{1, 2}\}$ and $\Delta^{3,4, -} \times \{ \epsilon_{3, 4} \}$ are connected, so to prove that $\mathfrak{X}_{0}^{s_1^+}(N_{1, 3})$ is connected,  it suffices to choose two points $(p, \epsilon_{1,2})$ with $p= (p_1, p_2 , p_3 , p_4)  \in \Delta^{1,2, -}$ and $(q , \epsilon_{3,4}) $ with $q = (q_1 , q_2 , q_3 , q_4) \in \Delta^{3,4, -}$ and show that they are connected by a path in $\mathfrak{X}_{0}^{s_1^+}(N_{1, 3}).$ Without loss of generality, we can assume that  $q_3 > q_1+q_2.$ By Section \ref{change_e0}, after doing a triangular switch $S_4$, the new coordinates (for the new triangulation $\cT '$) are given by the points 
$$p' = \left[ p_1,p_2,p_3, \displaystyle\frac{(-p_1-p_2+p_3)^2}{p_4} \right] \in \Delta^{1,2, -} \mbox{ and } \epsilon_p' =  \epsilon_{1, 2}$$
 and 
$$q' = \left[q_1, q_2, q_3, \displaystyle\frac{(q_1+q_2-q_3)^2}{q_4} \right] \in \Delta^{1,2, -} \mbox{ and } \epsilon_q' =  \{\epsilon_{1, 2}\}.$$ Since $\Delta^{1,2, -}_{\cT', \epsilon_{1, 2}}$ is connected, there is a path in $\mathfrak{X}_{\cT'} (N_{1,3} ) \cap \mathfrak{X}_{0}^{s_1^+}(N_{1, 3})$ between $p'$ and $q'.$ The image of this path by $S_4$ gives a path in $\mathfrak{X}_{0}^{s_1^+}(N_{1, 3})$ between $p$ and $q.$ So $\mathfrak{X}_{0}^{s_1^+}(N_{1, 3})$ is connected. Since the connectedness of $\mathfrak{X}_{+1}^{s_1^-}(N_{1, 3})$ and $\mathfrak{X}_{+1}^{s^+}(N_{1, 3})$ is proven in a very similar way, we will omit a detailed discussion.
\end{itemize}
\end{proof}

\section{Euler class $\mathrm{e}(\rho) = \pm 1$}\label{eul_1}

In this section we will discuss the case of $e(\rho) = \pm 1$ and prove Theorem \ref{eul_pm1} (1), (2) and (3).

We start with the following useful Proposition.

\begin{Proposition}\label{uncountable}
 The set of representations $[\rho] \in \mathfrak{X}_{\pm 1}(N_{1,3})$ such that all the balanced triangulations of $N_{1,3}$ are $\rho$--admissibile is of full-measure in $\mathfrak{X}_{\pm 1}(N_{1,3}).$
\end{Proposition}

\begin{proof}
  The proof is very close to the one of Proposition 5.5 of \cite{yan_ont}. We will discuss the case $[\rho] \in \mathfrak{X}_{+ 1}(N_{1,3})$, but the other case is very similar. Suppose $\rho$ is a type-preserving representation of $\pi_1(N_{1,3})$ with $e(\rho) = 1$ and $d$ is a decoration of $\rho.$ Let $\cT$ be a $\rho$--admissible balanced triangulation of $N_{1,3}$ with set of edges $E$ and set of triangles $T$, and let $(\lambda, \epsilon) \in \R_{>0}^E \times \{\pm 1\}^T$ be the coordinates of $[(\rho, d)]\in \mathfrak{X}^d_{1}(N_{1,3})$, that is, the $\lambda$--lengths of the edges and the signs of the triangles. Recall from Section \ref{back_tri} that given any balanced triangulation $\cT'$ there is a unique path in $\Upsilon$ connecting $\cT$ to $\cT'$, which corresponds to a sequence of triangle switches $\{S_i\}_{i =1}^n$ such that $\cT_0 = \cT$, $\cT_n = \cT'$ and $\cT_i$ is obtained from $\cT_{i-1}$ by doing the triangle switch $S_i.$ If $\cT_i$ is $\rho$--admissible and has $\lambda$--lengths $\lambda_i\in\R_{>0}^E$, then $\cT_{i+1}$ is $\rho$--admissible if and only if the Laurent polynomials describing the quantities $X^{(i+1)}_1$, $X^{(i+1)}_2$, $X^{(i+1)}_3$ and $X^{(i+1)}_4$ in terms of the quantities $X^{(i)}_1$, $X^{(i)}_2$, $X^{(i)}_3$ and $X^{(i)}_4$ (see Corollary \ref{change_e1}) are all non-zero. Using induction on $i$, we can see that $\cT'= \cT_n$ is $\rho$--admissible if and only if the quantities $X^{(n)}_1$, $X^{(n)}_2$, $X^{(n)}_3$ and $X^{(n)}_4$ are all non-zero, but these quantities are Laurent polynomials in terms of the variables $X^{(0)}_1$, $X^{(0)}_2$, $X^{(0)}_3$ and $X^{(0)}_4$, see formulas in Section \ref{change_e1}. The set of zeros $\mathcal{Z}_{\cT'}$ of these Laurent polynomials has zero Lebesgue measure because it is a Zariski-closed proper subset of $(\R_{>0})^4.$ Now since $\Upsilon$ is a countably infinite tree, there are countably many balanced triangulations $\cT'$ of $N_{1,3}.$ So the set $\cup_{\cT'}\mathcal{Z}_{\cT'}$ has zero Lebesgue measure as well. Hence the set $\mathcal{C} = (\R_{>0})^4\setminus \{\cup_{\cT}\mathcal{Z}_{\cT}\}$ has full measure. 
  
  Now, every $(x_1, x_2, x_3, x_4) \in \mathcal{C}$ gives a type-preserving representation in the following way. Consider a balanced triangulation $\cT$ of $N_{1,3}$ with set of edges $E$ and set of triangles $T$ and let $\epsilon \in \{\pm 1\}^T$ be a choice of signs for each triangles such that $\sum_{i= 1}^4\epsilon(t_i) = 2.$ Using the map $\Phi$ of Lemma \ref{lem_phi}, we get the length $\lambda \in (\R_{>0})^E$ using the formula:
 $$\lambda(e_{i,j}) = \displaystyle\frac{x_i x_j}{(x_1 x_2 x_3 x_4)^{\frac{1}{3}}}.$$
Therefore, $(\lambda, \epsilon)$ determines a decorated representation $(\rho, d)$ up to conjugation. By the calculation in Theorem \ref{cusp} one can see that this representation is type-preserving, and by (\ref{euler_sign}), $e(\rho) = 1.$ Since $(x_1, x_2, x_3, x_4) \in \mathcal{C},$ all the Laurent polynomials are non zero for all the balanced triangulations, and hence all balanced triangulations are $\rho$--admissible.
\end{proof}

In order to prove Theorem \ref{eul_pm1} (1), (2) and (3), we need to determine if the images of $2$-sided simple closed curves are sent to hyperbolic elements. Let $\rho \in \mathfrak{X}_{1} (N_{1,3} )$ and let $\cT$ be a balanced triangulation with triangles $\{t_1, t_2, t_3, t_4\}.$ We denote by $\gamma_{i,j}$ the $2$-sided curve associated to the edge $e_{i,j}$, which is the edge adjacent to triangles $t_i$ and $t_j.$ In what follows, we need to distinguish the curves corresponding to adjacent triangles with opposite signs, and the ones corresponding to adjacent triangles with the same sign. The following result, is a direct computation using Theorem \ref{SunYang}.

\begin{Lemma}[Trace of main curves]\label{curves_e1}
	Let $\rho$ and $\cT$ be as above, and let $\{i, j, k, l\} = \{1, 2, 3, 4\}.$ 
  \begin{enumerate}
    \item If $\epsilon(t_i)\neq \epsilon(t_j)$, then $|\mathrm{tr} \rho (\gamma_{i,j})| = \frac{|X_k^2 + X_l^2 - (X_i + X_j)^2|}{X_k X_l}.$
    \item If $\epsilon(t_i) = \epsilon(t_j)$, then $|\mathrm{tr} \rho (\gamma_{i,j})| = \frac{|X_k^2 + X_l^2 - (X_i - X_j)^2|}{X_k X_l}.$
  \end{enumerate}
  
Moreover, the right-hand sides of the six equations above are strictly greater than $2$ if and only if $(X_1, X_2, X_3, X_4)$ do not satisfy the generalized triangle inequalities \eqref{GTI}. 
\end{Lemma}

In other words, the six $2$-sided curves associated to a balanced triangulations are all sent to hyperbolic elements if and only if one of the coordinates $X_i$ is strictly greater than the sum of the other three coordinates, i.e.  $X_i > X_j + X_k + X_l.$

Now, using the formulas for the change of coordinates explained in Section \ref{change}, we can prove Theorem \ref{eul_pm1} (1), (2) and (3).

\begin{proof}[Proof of Theorem \ref{eul_pm1} (1) and (2)]

Let $\cT$ be a balanced triangulation of $N_{1,3}$ and let $\mathcal{C}$ be the full-measure set in $(\R_{>0})^4$ defined in the proof of Proposition \ref{uncountable}. Let $s = s_i^-$ and consider the set $\mathcal{R} \subset \mathfrak{X}_1^s (N_{1,3})$ of decorated type-preserving representations $(\rho , d)$ whose triangle coordinates are in $\mathcal{C}.$ This is a full-measure subset of $\mathfrak{X}_1^s (N_{1,3})$, and any $[\rho] \in \mathcal{R}$  is such that all balanced triangulations $\cT$ of $N_{1,3}$ are $\rho$--admissible. 

Let $(x_1, x_2, x_3, x_4) \in \mathcal{C}$ be the triangle coordinates of such a representation and $\varepsilon(t_i)$ the signs of the triangle. Using the proof of Lemma \ref{comp}, we get that $x_1, x_2, x_3, x_4$ does not satisfy inequalities \eqref{GTI}.  Note that moreover, if $x_i$ is the largest coordinate, then we have $\epsilon (t_i) = 1.$

In that case, Lemma \ref{curves_e1} shows that the absolute values of the traces of the six curves associated to the $6$ edges of $\cT'$ are strictly greater than $2.$ Every other balanced triangulation $\cT'$ of $N_{1,3}$ is obtained from $\cT$ by a sequence of triangle switches. Such triangle switches do not modify the sign $s$ of the punctures or the Euler class of the representation, and hence the new triangle invariants still do not satisfy inequalities \eqref{GTI}.

 Theorem \ref{eul_pm1} (1) and (2) is now proved because every $2$--sided simple closed curve in $N_{1,3}$ is the curve associated to some edge of some balanced triangulation $\cT'$ of $N_{1,3}.$ And hence, all $2$--sided simple closed curves are sent to hyperbolic elements.
\end{proof}

\begin{proof}[Proof of Theorem \ref{eul_pm1} (3)] 

We want to prove that for all type-preserving representations $\rho$ in $\mathfrak{X}_1^{s^+} (N_{1,3})$ or   $\mathfrak{X}_{-1}^{s^-} (N_{1,3})$, there exists a $2$-sided simple closed curve that is not sent to a hyperbolic element. It is sufficient to assume that $\rho \in \mathfrak{X}_1^{s^+} (N_{1,3})$, the other case being similar. Choose an arbitrary decoration of $\rho$, and let $\cT$ be a balanced triangulation of $N_{1,3}.$ To prove the theorem it suffices to find a triangulation $\cT'$, such that one of the $2$-sided curves associated to one of the edge of the triangulation, is sent to a non-hyperbolic element. The strategy for finding $\cT'$ is to use a \emph{trace reduction algorithm}, that is a rule which defines a sequence of balanced triangulations until we obtain a triangulation satisfying the above conditions.
	
	\emph{Trace Reduction Algorithm: }
	Let $\cT_0 = \cT$ and suppose that $\cT_n$ is obtained. 
	\begin{itemize}
		\item If $\cT_n$ is not $\rho$-admissible, then there is an edge $e$ of $\cT_n$ that is not $\rho$-admissible, and the element of $\pi_1 (N_{1, 3} )$ represented by the 2-sided simple closed curve in $\cT_n$ associated to $e$ is sent by $\rho$ to a parabolic element of $\PGL (2 , \R).$ So in that case the algorithm stops. 
		\item If $\cT_n$ is $\rho$-admissible, then we let $(\lambda , \epsilon)$ be the coordinates of $(\rho , d)$ in $\cT_n.$ If the quantities $X_1 , X_2 , X_3 , X_4$ satisfy the generalized triangle inequalities \eqref{GTI}, then by Lemma \ref{curves_e1}, an element of $\pi_1 (N_{1,3})$ is sent to a parabolic or elliptic element. So in that case the algorithm stops.
		\item Otherwise, there is a unique maximum among $X_1^{(n)}, X_2^{(n)},  X_3^{(n)}, X_4^{(n)}$, and in this case we let $\cT_{n+1}$ be the balanced triangulation obtained by doing a triangle switch along the triangle realizing the maximum. 
	\end{itemize}
	The result follows from the Lemma below.
\end{proof}

\begin{Lemma}
	The trace reduction algorithm stops in finitely many steps
\end{Lemma}

\begin{proof} Without loss of generality, we assume that the trace reduction algorithm didn't stop at step $n \in \N$ and that in $\mathcal T_n$ we have $X_4^{(n)}\geq X_3^{(n)}\geq X_2^{(n)}\geq X_1^{(n)}.$ Since $\rho$ is in $\mathfrak{X}_1^{s^+} (N_{1,3} ),$ by the proof of Theorem \ref{comp} we have that $\epsilon(t_4^{(n)})=-1$ and $\mathcal T_{n+1}$ is obtained from $\mathcal T_n$ by a switch along $t_4^{(n)}.$ In order to simplfy the notation, we will denote $X_i^{(n)} = X_i$ and $X_i^{(n+1)} = X_i'$, when there is no ambiguity.

By Lemma \ref{ch_e1}, we have the following expressions:

$$\begin{dcases}
		X_1' = \frac{-X_1+X_2+X_3}{X_4} X_1;\\  
		X_2'= \frac{X_1-X_2+X_3 }{X_4} X_2;\\
		X_3' = \frac{|X_1+X_2-X_3 |}{X_4} X_3; \\
		X_4'=  \frac{(-X_1+X_2+X_3)(X_1-X_2+X_3)|X_1+X_2-X_3|}{(X_4)^2}.
      \end{dcases}$$

We get that
\begin{equation*}
\begin{split}
X_1'+X_2'+X_3'\geq & \frac{-X_1+X_2+X_3}{X_4} X_1 + \frac{X_1-X_2+X_3 }{X_4} X_2 + \frac{-X_1-X_2+X_3 }{X_4} X_3\\
	=&\frac{(-X_1+X_2+X_3)(X_1-X_2+X_3)}{X_4}\\
= &\frac{(-X_1+X_2+X_3)(X_1-X_2+X_3)X_4}{(X_4)^2}\\
\geq &\frac{(-X_1+X_2+X_3)(X_1-X_2+X_3)|X_1+X_2-X_3|}{(X_4)^2} = X_4',\\
\end{split}
\end{equation*}
where the last inequality comes from $X_4 \geq X_1+X_2+X_3 \geq |X_1+X_2-X_3|.$

We now distinguish two cases:

\underline{Case I}: $X_1+X_2 \geq X_3.$

In that case, the computation above gives 
$$X_1'+X_2'-X_3' \geq X_4'.$$
By the symmetry of these formulas in $X_1,X_2,X_3$, we also have 
$$X_2'+X_3'+X_4' \geq X_1'$$
and 
$$X_1'+X_3'+X_4' \geq X_2'.$$

Hence the new coordinates $(X_1', X_2', X_3', X_4')$ satisfy inequalities \eqref{GTI} and the algorithm stops.

\underline{Case II}: $X_1+X_2 < X_3.$

By a direct computation, we have
$$X_1'+X_3'-X_2'=\frac{(X_3-X_2)^2-(X_1)^2}{X_4}>0,$$
where the inequality comes from $X_3-X_2>X_1.$ As a consequence, 
$$X_1'+X_3'+X_4'>X_1'+X_3' \geq X_2'.$$

Similarly we get $X_2'+X_3'-X_1' >0$, and hence $X_2'+X_3'+X_4' >X_1'.$

We have now three subcases:

\underline{Subcase 1}: $X_1'+X_2'+X_4' \geq X_3'.$

In this subcase, the new coordinates $(X_1', X_2', X_3', X_4')$ satisfy inequalities \eqref{GTI} and the algorithm stops.

\underline{Subcase 2}:  $X_1'+X_2'+X_4' < X_3'$ and $X_1'+X_2' \geq X_4'.$

In this subcase, the new coordinates fall in Case I above (exchanging the role of $X_3$ and $X_4$) and the trace reduction will finish at step $n+2.$

\underline{Subcase 3}: $X_1'+X_2'+X_4' < X_3'$ and $X_1'+X_2' < X_4'.$

In this subcase, the new coordinates will fall again in case II, (exchanging the role of $X_3$ and $X_4$). As a consequence, we see that the Trace reduction does not finish in a finite number of step if and only the coordinates are always in this subcase for each $n.$ Assuming this, we will now get a contradiction.

We note that:
\begin{align*}
	\frac{X_1^{(n+1)}+X_2^{(n+1)} }{X_3^{(n+1)}} =& \dfrac{\frac{-X_1^{(n)} + X_2^{(n)}+ X_3^{(n)} }{X_4^{(n)}} X_1 + \frac{X_1^{(n)}-X_2^{(n)}+X_3^{(n)} }{X_4^{(n)}} X_2 }{\frac{-X_1^{(n)}-X_2^{(n)}+X_3^{(n)} }{X_4^{(n)}} X_3} \\
			 \geq  &    \dfrac{(X_1^{(n)}-X_2^{(n)}+X_3^{(n)})}{(-X_1^{(n)}-X_2^{(n)}+X_3^{(n)}) } \frac{X_1^{(n)}+X_2^{(n)} }{X_3^{(n)} } \\
			\geq & \frac{X_1^{(n)}+X_2^{(n)} }{X_3^{(n)}},
\end{align*}
where the two inequalities follow from $X_3^{(n)} > X_1^{(n)} + X_2^{(n)} $ and $X_1^{(n)} \geq X_2^{(n)}.$

We project the sequence $X^{(n)} = (X_1^{(n)},X_2^{(n)},X_3^{(n)},X_4^{(n)})$ in the triangle $$\mathcal{D} =\{  (a,b,c) \in (0,1)^3 \, | \, a+b+c = 1 \}$$ by the following map:

$$\phi (X^{(n)}) = \dfrac{1}{(X_1^{(n)}+X_2^{(n)}+X_3^{(n)}+X_4^{(n)})} ((X_1^{(n)}+X_2^{(n)},X_3^{(n)},X_4^{(n)})) = (a_n, b_n, c_n)$$

The regions corresponding to the subcase 3 are the two regions:
$$\mathcal{L} = \{ (a,b,c) | c> a+b , b> a \} \mbox{   and   } \mathcal{R} = \{ (a,b,c) | b> a+c , c>a \} $$
These regions correpond to the grey regions on Figure \ref{fig:tracered1}. The trace reduction does not finish if $\phi(X^{(n)})$ alternates between these two regions.

    \begin{figure}
[hbt] \centering
\includegraphics[height=6 cm]{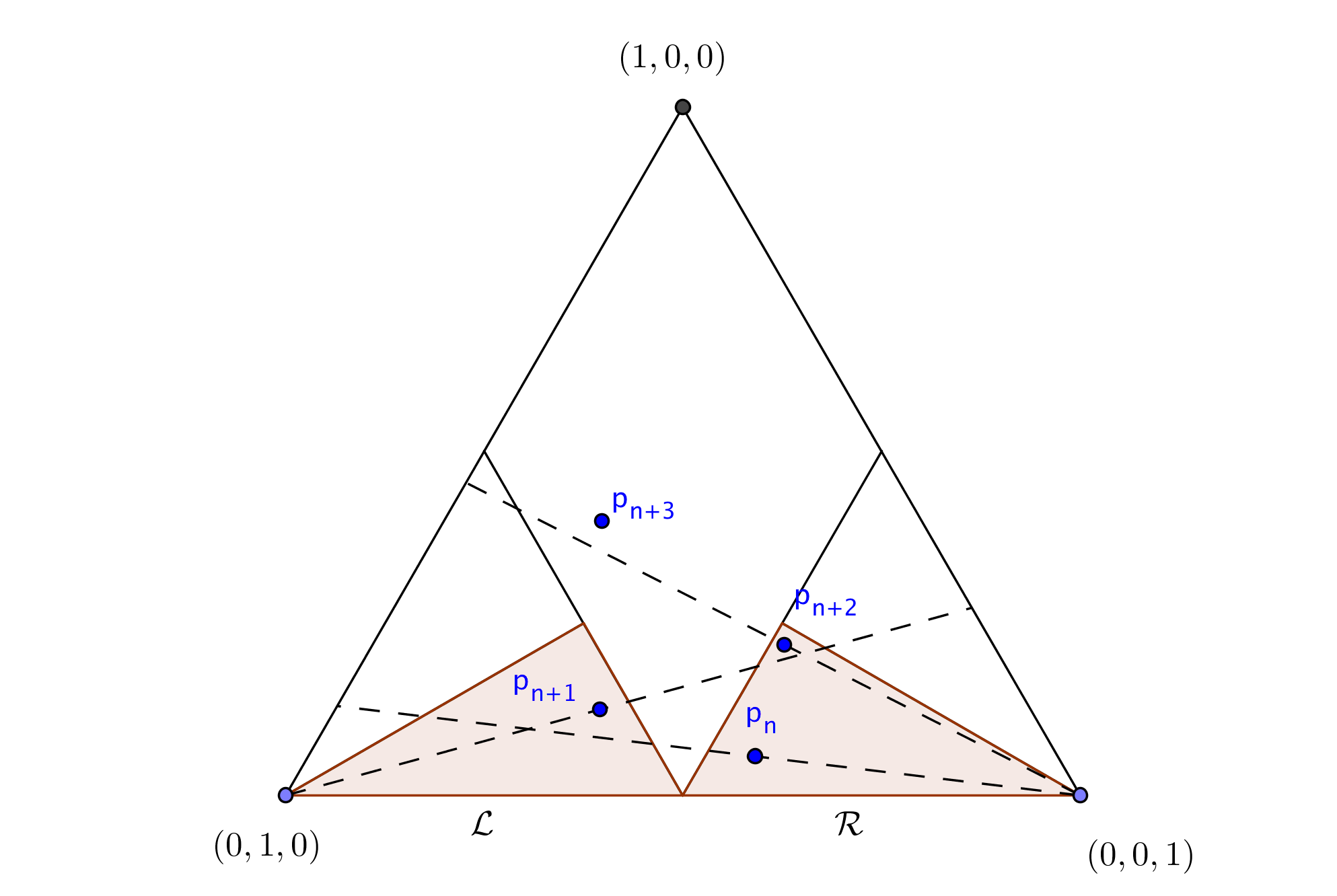}
\caption{Example of a sequence exiting $\mathcal{L}$ and $\mathcal{R}$ in three steps}
\label{fig:tracered1}
\end{figure}

Assume that $(a_n , b_n , c_n)$ is in $\mathcal{R}.$ The previous computation shows that $\frac{a_{n+1}}{b_{n+1}} > \frac{a_n}{b_n}.$ So the new point $(a_{n+1},b_{n+1},c_{n+1})$ is above the line joining the vertex $(0,0,1)$ of $\mathcal{D}$ to $(a_n, b_n, c_n)$, see Figure \ref{fig:tracered1}, but it's also in $\mathcal{L}.$ A simple geometric computation shows that the sequence $(a_n)$ is increasing and also that $\frac{a_{n+1}}{a_n} \geq \frac{1}{1-2a_n}  \geq \frac{1}{1-2a_0} > 1.$ Hence we get that $a_{n} > \left( \frac{1}{1-2a_0} \right)^n a_0$ for all $n \in \N.$ So we can see that there exists $n$ such that $a_n \geq \frac{1}{4} $ and hence $(a_n, b_n, c_n) \notin \mathcal{L} \cup \mathcal{R}$, which is a contradiction. This ends the proof of the Lemma.   
\end{proof}

\section{Euler class $\mathrm{e}(\rho) = 0$} \label{eul_0}

In this section we will discuss the case of $e(\rho) = 0$ and prove Theorem \ref{eul_pm1} (4). 

Let $\rho$ be a type-preserving representation of $\Gamma_{1,3}$ and let $d$ be a decoration of $\rho.$  Let $\cT$ be a $\rho$--admissible balanced triangulation of $N_{1,3}$ with set of edges $E$ and set of triangles $T$, and let $(\lambda, \epsilon) \in \R_{>0}^E \times \{\pm 1\}^T$ be the coordinates of $[(\rho, d)]\in \mathfrak{X}^d_{1}(N_{1,3})$, that is the $\lambda$--lengths of the edges and the signs of the triangles. We will follow the notation of Section \ref{eul_1}. 

If $\mathrm{e}(\rho) = 0$, by the formula \eqref{euler_sign} there are two triangles with positive sign and two triangles with negative sign. The relative signs of the triangles are determined by the relative signs of the punctures, according to Section \ref{sec:back}. Namely, if $\rho \in \mathfrak{X}_0^{s_i^{\pm}} (N_{1,3} )$, then there are two arcs in the triangulation that are not adjacent to the puncture $v_i$ and that do not belong to a common triangle. For each of these two arcs, the pair of triangles adjacent to the arc have the same sign, and the two pairs of triangles have opposite sign. All choices are equivalent for the computations and proofs.

Using the Trace Formulas (\ref{trace_formula}),  we obtain the following expressions for the (absolute value of the) trace of the six 2-sided curves associated with each arc in the triangulation $\cT,$ in terms of the triangle parameters.

\begin{Lemma}[Trace of main curves]\label{curves_e0}
  Let $\gamma_{i,j}$ be the distinguished simple closed curves corresponding to the arc  adjacent to the triangles $t_i$ and $t_j.$ Let $\{i, j, k, l\} = \{1, 2, 3, 4\}.$ 
  \begin{enumerate}
    \item If $\epsilon(t_i)\neq \epsilon(t_j),$ then $|\mathrm{tr} \rho (\gamma_{i,j})| =  \displaystyle\frac{\left(\epsilon(t_1)X_1+\epsilon(t_2)X_2+\epsilon(t_3)X_3+\epsilon(t_4)X_4 \right)^2 + 2 X_k X_l}{X_k X_l}.$
Moreover, we have $|\mathrm{tr} \rho (\gamma_{i,j})|>2.$
    \item If $\epsilon(t_i) = \epsilon(t_j),$ then $|\mathrm{tr} \rho (\gamma_{i,j}) |=  \displaystyle \left| \frac{\left(\epsilon(t_1)X_1+\epsilon(t_2)X_2+\epsilon(t_3)X_3+\epsilon(t_4)X_4 \right)^2 - 2 X_k X_l}{X_k X_l} \right|$.
Moreover, we have $|\mathrm{tr} \rho (\gamma_{i,j})) | \leqslant 2 $ if and only if $X_i, X_j , X_k , X_l$ satisfy the following set of inequalities:

\begin{equation} \label{sqGTI}
	\left\{ \begin{array}{ll}     
	\sqrt{X_k} &\leqslant  \sqrt{X_l} + \sqrt{X_i + X_j};\\
     \sqrt{X_l} &\leqslant  \sqrt{X_k} + \sqrt{X_i + X_j};\\          
     \sqrt{X_i + X_j} &\leqslant  \sqrt{X_k} + \sqrt{X_l}. \end{array} \right.
\end{equation}
  \end{enumerate}
\end{Lemma}

\begin{proof}
	The formulas for $|\mathrm{tr} \rho (\gamma_{i,j})|$ are direct applications of the Therem \ref{SunYang}, while the second part of $(1)$ is straightforward, so we only need to discuss the last statement in $(2),$ with $\epsilon(t_i)=\epsilon (t_j).$ Let $ p = \sqrt{X_k}, q = \sqrt{X_l}$ and $r = \sqrt{X_i + X_j}.$ We have
$$  -2 < \dfrac{(p^2+q^2-r^2)^2}{p^2q^2} -2 =\dfrac{(p+q+r)(p-q-r)(q-p-r)(r-p-q)}{p^2q^2}+2.$$

The absolute value of the above quantity is equal to $|\mathrm{tr} \rho (\gamma_{i,j})|.$

If $p>q+r,$ then $(p-q-r)(q-p-r)(r-p-q)>0$ and hence the left-hand side of the equation is greater than $2.$ The case $q > p+r$ and $r> p+q$ are similar.

If $ |\mathrm{tr} \rho (\gamma_{i,j}) | > 2,$ then $(p-q-r)(q-p-r)(r-p-q)>0$ and one of the inequalities in \eqref{sqGTI} is satisfied.
\end{proof}

	\begin{Remark} We don't need the actual sign of the new triangles but only their relative signs. 
	\end{Remark}

We can now prove the main theorem of this section. 

\begin{Theorem}\label{thm_eul_0}
  Every non-elementary type-preserving representation  $\rho\co\pi_1(N_{1,3})\to \PGL(2,\R)$ with relative Euler class $e(\rho)= 0$ sends some non-peripheral simple closed curve to an elliptic or parabolic element. More precisely, if $\rho \in \mathfrak{X}_0^{s_i^{\pm}} (N_{1,3} ),$ then all $2$-sided simple closed curve that separate the two punctures $v_j$ and $v_k,$ where $\{i,j,k \} = \{ 1, 2 , 3 \},$ are sent to hyperbolic elements, and there exists some $2$-sided simple closed curve that does not separate $v_j$ and $v_k$ which is sent to a non-hyperbolic element.
\end{Theorem}

\begin{proof}
	It is sufficient to assume that $\rho$ is in $\mathfrak{X}_0^{s_1^+} (N_{1,3}).$ (The other cases can be proved similarly.) Choose an arbitrary decoration $d$ of $\rho,$ and let $\cT$ be a balanced triangulation of $N_{1,3}.$ To prove the first part of the theorem it suffices to find a triangulation $\cT',$ such that one of the $2$-sided curves associated to one of the edge of the triangulation, is sent to a non-hyperbolic element. The principle of the proof relies on a Trace reduction algorithm, similar to the one used in the proof of Theorem \ref{eul_pm1} (3). 

	\emph{Trace Reduction Algorithm:}
	Let $\cT_0 = \cT$ and suppose that $\cT_n$ is obtained. 
	\begin{itemize}
		\item If $\cT_n$ is not $\rho$-admissible, then there is an edge $e$ of $\cT_n$ that is not $\rho$-admissible, and the element of $\pi_1 (N_{1, 3} )$ represented by the 2-sided simple closed curve in $\cT_n$ associated to $e$ is sent by $\rho$ to a parabolic element of $\PGL (2 , \R).$ So in that case the algorithm stops. 
		\item If $\cT_n$ is $\rho$-admissible, then we let $(\lambda , \epsilon)$ be the coordinates of $(\rho , d)$ in $\cT_n.$ If the quantities $X_1 , X_2 , X_3 , X_4$ satisfy the  inequalities \eqref{sqGTI}, then by Lemma \ref{curves_e0}, an element of $\pi_1 (N_{1,3})$ is sent to a parabolic or elliptic element. So in that case the algorithm stops.
		\item Otherwise, there is a unique maximum among $X_1^{(n)}, X_2^{(n)},  X_3^{(n)}, X_4^{(n)},$ and in this case we let $\cT_{n+1}$ be the balanced triangulation obtained by doing a triangle switch along the triangle realizing the maximum. 
	\end{itemize}
	The result follows from the Lemma below.
\end{proof}
	
	\begin{Lemma}
	The trace reduction algorithm stops in finitely many steps
	\end{Lemma}
	
	\begin{proof}
		Assume by contradiction that we have an infinite sequence $\cT_n$ constructed by the trace reduction algorithm. Let $(a_n , b_n , c_n , d_n)$ be the normalized variables:
		$$(a_n, b_n, c_n , d_n) = \frac{1}{X_1^{(n)}+ X_2^{(n)}+  X_3^{(n)}+ X_4^{(n)} } (X_1^{(n)}, X_2^{(n)},  X_3^{(n)}, X_4^{(n)}),$$
		so that $a_n + b_n + c_n + d_n = 1.$ We let 
		\begin{align*}
			h_n & = \max \{ \sqrt{a_n} - \sqrt{b_n} - \sqrt{c_n+d_n} , \sqrt{b_n} - \sqrt{a_n} - \sqrt{c_n+d_n}, -\sqrt{a_n} - \sqrt{b_n} + \sqrt{c_n+d_n} \}, \text{ and}\\
			k_n & =  \max \{ \sqrt{c_n} - \sqrt{d_n} - \sqrt{a_n+b_n} , \sqrt{d_n} - \sqrt{c_n} - \sqrt{a_n+b_n}, -\sqrt{c_n} - \sqrt{d_n} + \sqrt{a_n+b_n} \}.
		\end{align*}
		By hypothesis, the quantities $(a_n , b_n , c_n , d_n)$ do not satisfy the inequalities \eqref{sqGTI}, so for all $n \in \N$ we have that $k_n > 0$ and $h_n> 0.$ We will derive a contradiction in several steps.
		
		\textbf{Step 1}: \textit{For all $n \in \N,$ $\max (a_n , b_n , c_n , d_n ) > \frac 12.$} 
		
		Assume without loss of generality that $d_n = \max (a_n , b_n , c_n , d_n ).$  
		
		In that case, $h_n = \sqrt{c_n+d_n}  -\sqrt{a_n} -\sqrt{b_n} > 0 ,$ by hypothesis. So $\sqrt{c_n+d_n} > \sqrt{a_n} + \sqrt{b_n},$ and hence $\sqrt{c_n} + \sqrt{d_n} > \sqrt{a_n + b_n}.$ From this we deduce that $k_n = \sqrt{d_n} - \sqrt{c_n} - \sqrt{a_n + b_n}.$  By hypothesis, we have that $k_n > 0$ so $\sqrt{d_n} > \sqrt{c_n} + \sqrt{a_n + b_n} .$ 
		
		Using this last inequality we get $d_n > a_n + b_n + c_n = 1-d_n$ and hence $d_n > \frac 12.$
		
		\textbf{Step 2}: \textit{The sequence defined by $u_n = |a_n + b_n - \frac 12 | = |c_n + d_n - \frac 12|$ is decreasing.}
		
		Assume without loss of generality that $d_n =\max (a_n , b_n , c_n , d_n ),$ so $c_n + d_n > \frac 12.$ It suffices to show that:
		$$ 1- (c_n + d_n) < c_{n+1} + d_{n+1} < c_n + d_n.$$
		
		We first prove the right inequality. From Lemma \ref{ch_e0}, we obtain:
		$$ (a_{n+1} , b_{n+1} , c_{n+1} , d_{n+1} ) = M_n \left( a_n , b_n , c_n , \dfrac{(a_n+b_n-c_n)^2}{d_n}, \right)$$
		with $M_n = \dfrac{1}{a_n + b_n + c_n + \frac{(a_n+b_n-c_n)^2}{d_n}} > 1.$ From this we obtain that $a_{n+1} > a_n,$ $b_{n+1} > b_n$ and $c_{n+1}> c_n.$ Hence in this case $c_n+d_n = 1-(a_n + b_n) < 1- (a_{n+1}+b_{n+1}) = c_{n+1} + d_{n+1}$ and the right inequality is proven.
		
		For the left inequality, we consider the quantity:
		\begin{align*} 
			a_n+b_n+a_{n+1}+b_{n+1} 
				& = a_n + b_n + \dfrac{a_n + b_n}{a_n + b_n + c_n + \frac{(a_n+b_n-c_n)^2}{d_n} } \\
				& = 1 - (c_n + d_n) + \dfrac{d_n ( 1- c_n - d_n)}{d_n(1-d_n) + (1-2c_n-d_n)^2} \\	
				& = 1 - \dfrac{4c_n^3+4c_n^2 (2d_n-1) + c_n (2d_n-1)^2}{d_n(1-d_n) + (1-2c_n-d_n)^2 } < 1 .
		\end{align*}
		The last term is lower than $1$ because $d_n > \frac 12.$
		
		Hence $c_{n+1} + d_{n+1} > 1 - (c_n + d_n)$ and the left inequality is also proven.
		
		\textbf{Step 3}: \textit{$\displaystyle \lim_{n\to \infty} u_n = 0.$}
		
		The sequence $(u_n)$ is decreasing and bounded below by $0$ so it converges towards a fixed point. 
		
		Assume that $u_{n+1} = u_n$ and that $d_n = \max (a_n , b_n, c_n, d_n).$ Then $|c_n + d_n - \frac 12| = |c_{n+1} + d_{n+1} - \frac 12 |.$ There are two cases:
		\begin{itemize}
			\item If $c_n + d_n - \frac 12= c_{n+1} + d_{n+1} - \frac 12$ then we obtain after computation
			$$(c_n+d_n-1)(1-2c_n)(1-2(c_n-d_n))=0.$$
			Hence we get $c_n + d_n = 1$  or $c_n + d_n = \frac 12.$
			\item if $c_n+d_n - \frac 12 = \frac 12 - (c_{n+1} + d_{n+1})$ then we obtain after computation
			$$c_n (1-2(c_n+d_n))^2  = 0.$$
			Hence we get $c_n = 0$ or $c_n + d_n = \frac 12.$
		\end{itemize}
		
		The only fixed points are $0$ and $\frac 12.$ But the sequence $(u_n)$ is decreasing and bounded above by $\frac 12$ so it's either constant and equal to $\frac 12$ or converging towards $0.$ If $u_n = \frac 12$ then it means that at least two among $a_n , b_n , c_n , d_n$ are zero which is excluded.

		\textbf{Step 4}: \textit{$\displaystyle \lim_{n\to \infty}  \max (a_n , b_n) = 1/2 = \lim_{n\to \infty} \max (c_n , d_n).$ Similarly, $\lim_{n\to \infty} \min (a_n , b_n) = 0 = \lim_{n\to \infty} \min (c_n , d_n).$}
		
		Let $\varepsilon >0.$ Then for $n$ large enough $u_n < \varepsilon.$ Suppose  without loss of generalities, that for such an $n$ we have $d_n = \max (a_n , b_n , c_n , d_n).$  As $d_n > \frac 12,$ we have easily that $d_n < \frac 12 + \varepsilon$ and that $c_n < \varepsilon.$ Similarly, as $\sqrt{c_n+d_n} > \sqrt{a_n}+\sqrt{b_n}$ we have that $\sqrt{a_n b_n}< \varepsilon$ so that $\min (a_n , b_n) < \varepsilon.$ This gives the desired result.
		
		\vspace{.5in}
		
		We can now derive the final contradiction from this last fact. Let $\epsilon >0$ be small enough and $n_0$ such that for all $n > n_0$ we have $u_n < \epsilon.$
				
		 Assume without loss of generality that $b_{n_0} < \varepsilon$ and $c_{n_0}< \varepsilon.$ Then we see that for all $n > n_0,$ $b_n$ and $c_n$ are always the minimum and hence are converging towards $0.$ On the other hand, the formulas show that the sequences $b_n$ and $c_n$ should be increasing which is a contradiction.	\end{proof}

     \section{Ergodicity}\label{sec:ergodicity}

     In this section, we consider the ergodicity of the action of the mapping class group on different components of $\mathfrak X(N_{1,3}).$

\subsection{Euler class  $e(\rho)=0.$}

In this section we will prove the following theorem.

     \begin{Theorem}
     	The pure mapping class group action is ergodic on each component of $\mathfrak{X}_0 (N_{1,3} ).$
     \end{Theorem}

     \begin{proof}
     We prove that the action of the mapping class group is ergodic on $\mathfrak{X}_0^{s_1^+} (N_{1,3} ).$ The other five components can be done in the same way by symmetry. Choose a balanced triangulation $\cT.$ From Section \ref{sec:conn} we know that $ (\Delta^{1,2,-} \times \{ \epsilon_{1,2} \}) \cup (\Delta^{1,2,+} \times \{ \epsilon_{3,4} \} ) $ is diffeomorphic to an open dense subset of $\mathfrak{X}_0^{s_1^+} (N_{1,3} ).$ So we can identify 
     $$\Delta^{1,2} = \{(x_1,x_2,x_3,x_4) \in \Delta | x_1+x_2 \neq x_3+ x_4 \} $$ 
     as an  open  and dense subset of $\mathfrak{X}_0^{s_1^+} (N_{1,3} ).$  For convenience, we consider a different embedding $i : \Delta \rightarrow \R_{>0}^4$ given by 
     $$i((x_1 , x_2 , x_3 , x_4)) = (\frac{x_1}{x_1+x_2},\frac{x_2}{x_1+x_2}  , \frac{x_3}{x_1+x_2} , \frac{x_4}{x_1+x_2}).$$
     Denote by $\Omega_{1,2}$ the image $i(\Delta^{1,2}).$ Then 
     $$\Omega_{1,2} = \{ ( a , 1-a , c , d ) \in (0,1)^2\times \R_{>0}^2 \, | \, c+d \neq 1 \}.$$

     We prove ergodicity by proving that any measurable function $F\co \Omega_{1,2}\rightarrow \R$ which is invariant by the group generated by triangle switches is almost everywhere constant. Let $\rho \in \mathfrak{X}_0^{s_1^+} (N_{1,3}).$ Using Theorem \ref{thm_eul_0}, there exists a $2$-sided simple closed curve which is sent to a non-hyperbolic element. Up to the mapping class group action, we can assume that this curve is the curve $\gamma_{1,2}$ associated to the triangulation $\cT.$  So we have that 
     $$|\mathrm{tr} (\rho' (\gamma_{1,2} ) ) | = \dfrac{(a+b-(c+d))^2 - 2 cd}{cd} < 2.$$
 This implies that the image of $\rho$ in $\Delta_{1,2}$ is in the domain 
     $$L = \{ (a,1-a, c, d) \in \Omega_{1,2} , (c+d-1)^2 < 4 cd \}.$$
The action of the triangle switches $S_3$ and $S_4$ on $\Omega_{1,2}$ are given by 
     $$S_3 (a , 1-a, c , d) = \left( a , 1-a , \frac{(d-1)^2}{c} , d \right) \quad \mbox{ and } \quad S_4 (a , 1-a, c , d) = \left( a , 1-a , c, \frac{(c-1)^2}{d} \right).$$
     We can see that the sets 
     $$E_{a,k} = \{ ( a, 1-a , c , d ) \in \Omega_{1,2} | (c+d-1)^2=(k+2) cd \} $$
     are invariant by $S_3$ and $S_4.$  For $k \in (-2 , 2),$ such a set is an ellipse. If we denote by $\tau_{3,4} = S_3 \circ S_4,$ the Dehn twist along the curve $\gamma_{3,4};$ we get that $\tau_{3,4}$ acts on $E_{a,k}$ as a rotation of angle $\theta_k$ depending only on $k.$ Hence, for almost every $k \in (-2 , 2)$ the $\langle \tau_{3,4} \rangle$-action is ergodic on $E_{a,k}.$ 

     As the ellipse $E_{a, k}$ is transverse to the line $c+d=1$  at $(0,1)$ and $(1,0)$, there exists points in the $\tau_{3,4}$-orbit of $\rho$ in the domain $\{ ( a, 1-a , c , d ) \in \Omega_{1,2} | |2a-1|< c+d <  1 \}.$ Hence up to the action of the mapping class group, we can assume that $\rho$ satisfies this condition.
     
     Let $p = (a,1-a , c , d)$ be any point in the domain $R$ defined by
     $$R = \left\{ (a,1-a , c , d) \, | \, (c+d-1)^2 < 4 cd   \mbox{ and } |2a-1| <c+d< a+b \right\}.$$
     The point $S_1 (p)=  \displaystyle i \left( M \left( \frac{(c+d-(1-a))^2}{a} , 1-a , c , d \right) \right)$ belongs to a unique curve $E_{a', k'}$ which is invariant by the element $\tau_{3,4}.$ We denote by $Q_1 = S_1 (E_{a',k'} )$ and we note that $p \in Q_1.$ Moreover, the element $\tau_{1,3} \tau_{4,1} = S_1 \tau_{3,4} S_1$ preserves the curve $Q_1.$ So we infer that 
     \begin{align*} 
     	k'+2 & =\dfrac{\left(c'+d'-1 \right)^2}{c'd'}  \dfrac{\left(c+d-(\frac{(c+d-(1-a))^2}{a} + (1-a) )\right)^2}{cd} \\
     		&= \dfrac{(c+d-1)^2}{cd} \dfrac{(c+d-(1-a))^2}{a^2} = k \dfrac{(c+d-(1-a))^2}{a^2}.
     \end{align*}

     As $|2a-1| < c+d < 1$ because $p \in R$,  we have that $|(c+d-(1-a))| < a$ and hence $k' \in (-2, 2).$ This implies that the action of $\tau_{3,4}$ is ergodic on $E_{a', k'}$ for almost every $k'$ and consequently the action of $\langle \tau_{1,3} \tau_{4,1} \rangle$ is ergodic on $Q_1$ for almost every point $p \in R.$ With a similar argument, we prove that that the group generated by the element $ \tau_{2,3} \tau_{4,2} $ acts ergodically on a  curve $Q_2$ passing through $p$ for almost every $p \in R.$

     The gradients of the three curves $E_{a,k}, Q_1 , Q_2$ span the tangent space of $\Omega_{1,2}$ at $p.$ By ergodicity of the action on each of these three curves, we deduce that the restriction of $F$ to the domain $R$ is almost everywhere constant. This concludes the proof of ergodicity.
     \end{proof}

\subsection{Euler class $e(\rho) = \pm 1$}\label{ergodicity2}

In this section we will prove the following result.

     \begin{Theorem}
     	Let $s^+=(+1,+1,+1)$ and $s^+=(-1,-1,-1),$  then the mapping class group action is ergodic on the connected components $\mathfrak{X}_{1}^{s^+} (N_{1,3} )$ and $\mathfrak{X}_{-1}^{s^-} (N_{1,3} ).$
     \end{Theorem}

     \begin{proof} We will use a different parametrization of the character variety, and we will work with trace coordinates of matrix representatives. Let $\cT$ be a triangulation. We choose a presentation $$\pi_1 (N_{1,3}) = \{ \g_1 , \g_2 , \g_3 , \g_4 | \alpha_1 \alpha_2 \alpha_3 \alpha_4 \},$$ where the generators $\g_1 , \g_2 , \g_3 , \g_4$ are the four one-sided curves associated to $\cT$, see Figure \ref{fig:onesided}. We can define a representative in $\mathrm{Aut} (\pi_1 (N_{1,2} ))$ of one of these involution as 
     $\theta_4 (\g_i) = \g_i^{-1}$ for $i \neq 4$ and $\theta_4 (\alpha_4) = \g_3 \g_2 \g_1.$ The other involutions are obtained by a cyclic permutation of indices $(1,2,3,4).$ The product of two of these involutions corresponds to a Dehn twist about the $2$--sided simple closed curves in $N$ associated to the pair of $1$--sided curves given by the two involutions. In addition, these involutions corresponds to the triangle switches discussed in Section \ref{back_tri}. 

     Let $\rho \co \pi_1 (N_{1 , 3}) \rightarrow \mathrm{PGL}(2, \R)$ be a type-preserving representation with Euler class $e (\rho) = 1.$ As $\pi_1 (N_{1,3})$ is a free group, one can choose a lift $\widetilde{\rho}$ of the representation into $\mathrm{ISL}(2 , \R) = \mathrm{SL}(2 , \R) \sqcup i \mathrm{SL}^-(2, \R) \subset \mathrm{SL}(2, \C)$
     where orientation reversing isometries are sent to elements of $i \mathrm{SL}^-(2, \R) .$ We denote by $A, B, C \in i \mathrm{SL}^-(2, \R)$ the elements $\widetilde{\rho}(\alpha) , \widetilde{\rho}(\beta) , \widetilde{\rho}(\gamma).$ The three boundary component of $N_{1,3}$ are represented by $\alpha\beta$, $\beta \gamma$ and $\gamma \delta$ in $\pi_1 (N_{1,3}).$ We denote by $X, Y , Z$ the matrices in $\mathrm{SL}(2 , \R)$ corresponding to $AB$, $BC$ and $CA.$

     We denote by $(ia,ib,ic,id, x, y, z) \in (i\R)^4 \times \R^3$ the traces of the matrices $A,B,C,D,X,Y,Z$, where $a,b,c,d,x,y,z$ are real numbers. As $\rho$ is a representation of the free group in three generators into $\mathrm{SL}(2,\C)$, the coordinates satisfy the following equation 

     \begin{equation}\label{eqn_trpgl}
     -a^2-b^2-c^2-d^2+x^2+y^2+z^2+(ab+cd)x+(ad+bc)y+(ac+bd)z+abcd+xyz-4=0.
     \end{equation}

     As $\rho$ is a type-preserving representation, we have that $x ,y ,z = \pm 2.$ To determine the relative sign of $x$, $y$ and $z$, we consider $S'$ the embedded four-holed sphere obtained by cutting the surface $N_{1,3}$ open along the curve $\alpha.$ This induces a representation $\rho'\co \pi_1 (S') \rightarrow \mathrm{PSL}(2 , \R)$, and we have the equality $e(\rho) = e(\rho').$  The boundary components of $S'$ are $\alpha^2 , \alpha\beta , \beta\gamma$, and $\gamma \alpha.$ The traces of the images of these boundary curves by the representation $\rho'$ are $(-a^2 -2 , x,y,z).$

     As $e(\rho') = \pm1$, following Benedetto-Goldman\,\cite{ben_the}, we have that $(-a^2-2)xyz < 0$, which implies that $xyz >0.$ The different lifts of $\rho$ differ by an action of the central character which multiply the image of one of the generator by $-I.$ For example, choosing $-A$ instead $A$ changes the trace coordinates as $(ia, ib, ic, id, x, y, z) \mapsto (-ia, ib, ic, -id, -x, y, -z).$

     Therefore, up to choosing a different lift, we can assume without loss of generalities that  $x = y = z = 2.$ Equation \eqref{eqn_trpgl} becomes 
     $$abcd+16= a^2+b^2+c^2+d^2-2(ab+ac+ad+bc+bd+cd).$$
     This equation can be written in one of the three equivalent ways that we will use :
     \begin{align*}
     	(a+b-(c+d))^2 & =(ab+4)(cd+4), \\
     	(a+c-(b+d))^2 & = (ac+4)(bd+4), \\
     	(a+d-(b+c))^2 & = (ad+4)(bc+4).
     \end{align*}

     Theorem \ref{eul_pm1} (4) implies that up to the action of the mapping class group, at least one of the six curves $\alpha \beta^{-1}, \beta \gamma^{-1} , \gamma \alpha^{-1}, \alpha^2 \beta \gamma, \beta \alpha \beta \gamma$ and $ \alpha \beta \gamma^2$ has its trace in the interval $(-2 , 2).$ The traces of these six $2$-sided curves are given by $-(ab+2), -(cd+2), -(bc+2), -(ad+2), -(ac+2)$ and $-(bc+2)$ respectively.

     By symmetry of the problem in the different generators, we can assume without loss of generality that $cd+4 \in (0,4)$, and so $c$ and $d$ have opposite signs. In that case, the set 
     $$E_{c,d} = \{ (x, y, z, w) \in \R^4 | (x+y-(z+w))^2=(xy+4)(zw+4), z= c, w=d \}$$
     is an ellipse that is invariant by the action of the involutions $\theta_a$ and $\theta_b.$ The composition $\theta_a \circ \theta_b$ is the Dehn twist along the curve $\alpha \beta \gamma^2$, and acts as a rotation of angle $\theta_{a,b}$ depending only on $a$ and $b.$ This means that the group generated by this Dehn twist acts ergodically on the ellipse $E_{c,d}$ for almost every value of $a, b.$ The orbit of the representation will be dense in $E_{c,d}.$

     This ellipse intersects the hyperplane defined by $a=0$ transversally. So, there is an open region intersecting the ellipse $E_{c,d}$ such that $a$ is small enough and of the desired sign, so that $ab +4 \in (0,4) .$ With the same reasoning, this means that the Dehn twist $\theta_c \circ \theta_d$ acts ergodically on the ellipse $E_{a,b}$ for almost every value of $c, d$ in this open region.

     As $c$ and $d$ are of opposite signs and $a$ is small enough, we can also assume that one of $ac+4$ or $ad+4$ is also in $(0,4).$ Without loss of generality, assume that $ac+4 \in (0,4)$ . Once again, it means that the Dehn twist $\theta_b \circ \theta_d$ acts ergodically on the ellipse $E_{a,c}$ for almost every value of $b, d$ in this open region. 

     As the orbit of the representation $\rho$ is dense in $E_{c,d}$, we can assume that $\rho$ belongs to this open region. So we have three different $2$-sided simple closed curves that are sent to elliptic element. For almost all representations, each of the three Dehn twist along these three curves will act ergodically on the corresponding ellipses passing through $\rho.$  These three ellipses are transverse at $\rho$, as they are located in three different two-dimensionnal planes of $\R^4.$ Hence the gradients of these ellipses generate the tangent space at $\rho.$ This implies ergodicity of the mapping class group action on the entire set.
     \end{proof}

     At this point we were not able to prove the ergodicity of the action of the mapping class group on the components $\mathfrak X_1^s(N_{1,3})$ where $s$ contains exactly two $+1$'s and $\mathfrak X_{-1}^s(N_{1,3})$ where $s$ contains exactly two $-1$'s, which we tend to believe is true and deserves further studies.

\end{document}